\documentclass[11pt]{article}
\usepackage{amssymb, amsmath, amsfonts, tikz,  ytableau}
\usepackage{graphicx}
\usepackage{subfigure}
\usepackage{amsthm}
\usepackage[indent,margin=1cm]{caption}                 
\usepackage{float}
\usepackage[a4paper, margin=1in]{geometry}                   \usepackage{extpfeil}
\usepackage[colorlinks,linkcolor=blue,anchorcolor=blue,citecolor=blue]{hyperref}
\usepackage{diagbox}
\usepackage{pgfplots}
\usepackage{mathrsfs}
\usepackage{ytableau}
\usetikzlibrary{calc,through,backgrounds}
\usepackage{multicol}
\usepackage{enumitem}
\usetikzlibrary{arrows}

\usetikzlibrary{arrows.meta}
\usetikzlibrary{positioning, calc, decorations.pathreplacing}

\numberwithin{equation}{section}
\numberwithin{figure}{section}

\hypersetup{colorlinks = true}
\newtheorem{thm}{Theorem}[section]

\newtheorem{cor}[thm]{Corollary}
\newtheorem{lem}[thm]{Lemma}

\newtheorem{prop}[thm]{Proposition}

\newcommand{\pattern}[4]{%
    \raisebox{0.6ex}{%
        \begin{tikzpicture}[scale=0.35, baseline=(current bounding box.center), #1]
        \foreach \x/\y in {#4} {
            \fill[gray!60] (\x,\y) rectangle + (1,1);
        }
        \pgfmathtruncatemacro{\n}{#2}
        \foreach \i in {0,...,\n} {
            \ifnum\i=0
                \draw (\i,0) -- (\i,\n);
                \draw (0,\i) -- (\n,\i);
            \else\ifnum\i=\n
                \draw (\i,0) -- (\i,\n);
                \draw (0,\i) -- (\n,\i);
            \else
                \draw[black] (\i,0) -- (\i,\n);
                \draw[black] (0,\i) -- (\n,\i);
            \fi\fi
        }
        \foreach \x/\y in {#3} {
            \draw[black] (\x,\y) -- (\x+1,\y+1);
            \draw[black] (\x+1,\y) -- (\x,\y+1);
        }
        \end{tikzpicture}%
    }%
}

\allowdisplaybreaks

\begin{document}
	\begin{center}
		{\bf \Large Lower and upper bounds of Schur characters}
	\end{center}

    \begin{center}
	{\bf Zhuowei Lin$^1$, Candice X.T. Zhang$^2$ and Zhong-Xue Zhang$^3$      \\[6pt]}
	
	{\it$^{1}$~Center for Combinatorics, LPMC\\
		Nankai University, Tianjin 300071, P. R. China\\[8pt]
		$^{2}$~School of Mathematical Sciences, \\
		Tianjin University of Technology, Tianjin 300384, P. R. China\\[8pt]
        $^{3}$~Mathematics Teaching and Research Section, Basic Department\\
        Naval University of Engineering, Wuhan, Hubei 430030, P. R. China\\
		
		Email: $^1${\tt zwlin0825@163.com},\ \ $^2${ \tt zhang.xutong@foxmail.com},\ \ $^3${\tt 2520252129@nue.edu.cn}}
    \end{center}
	\noindent\textbf{Abstract.} 
    Characterizations of lower and upper bounds for dual characters of flagged Weyl modules have attracted considerable interest. In this paper, we establish explicit pattern avoidance characterizations for lower and upper bounds of Schur characters, namely, the dual characters of Weyl modules associated with arbitrary diagrams. This setting extends the corresponding extremal problems for dual characters of flagged Weyl modules and includes skew Schur polynomials and, through Rothe diagrams, Stanley symmetric functions.
    Specifically, for a permutation $w$, we show that the Stanley symmetric function $F_w$ attains the lower bound if and only if $w$ avoids $321,2143,2413,3142$, and $3412$, and attains the upper bound if and only if $w$ avoids $312$ and $321$.
   

	\noindent \emph{AMS Mathematics Subject Classification 2020: }05A05, 05E05 
	
	\noindent \emph{Keywords: Schur character, Weyl module, lower bound, upper bound} 
	\section{Introduction}

    The purpose of this paper is to investigate the lower and upper bounds of Schur characters, also known as the dual characters of Weyl modules. Let $D$ be a diagram contained in $[n]^2$, i.e., the $n\times n$ grid. 
    Equivalently, $D$ can be viewed as an ordered list of subsets $D=\left(D_1,D_2,\ldots,D_n\right)$, where $D_j=\{i\mid (i,j)\in D \}$ consists of all row indices of boxes in the $j$-th column of $D$. 
    The Schur character, denoted by $\chi_D$, expands into a sum of monomials with nonnegative integer coefficients. 
    
    For the case where every coefficient is either $0$ or $1$, these polynomials are called zero-one or multiplicity-free, and they may correspond to certain geometric objects. 
    Fink, M\'esz\'aros and St. Dizier~\cite{fink2021zero} gave the first pattern-avoidance characterization of zero-one Schubert polynomials by applying Magyar's orthodontia formula for Schubert polynomials~\cite{magyar1998schubert, mestrans}. 
    The correspondence between zero‑one Schubert polynomials and multiplicity‑free matrix Schubert varieties implies that all information about Grothendieck polynomials is encoded in their zero‑one counterparts. For Grothendieck polynomials, Chen, Fan and Ye~\cite{CFY} proved that such a polynomial is zero-one if and only if it avoids six specific patterns. For zero-one key polynomials, Hodges and Yong~\cite{hodges2023multiplicity} provided a criterion (first announced in~\cite[Theorem 4.10]{HY-1}) using the quasi-key model of Assaf and Searles~\cite{AS-1} together with the Kohnert diagram model~\cite{Koh}. 
    More generally, M\'esz\'aros, St. Dizier and Tanjaya~\cite{meszaros2021principal} studied zero-one dual characters of flagged Weyl modules and proposed a conjectural criterion, which was later confirmed by Guo, Lin and Peng~\cite{GLP}. In fact, their results unified the aforementioned results on Schubert polynomials and key polynomials.

	For the upper bounds of the dual characters of flagged Weyl modules, 
    Peng, Lin and Sun~\cite{PengLinSun2024} gave the sufficient and necessary conditions for these polynomials to attain their upper bounds.
    This characterization confirms a conjecture first proposed by M\'esz\'aros, St. Dizier, and Tanjaya~\cite{meszaros2021principal} in their study of the principal specialization of these dual characters.
    The work of Peng et al.~\cite{PengLinSun2024} not only involves the upper bounds results for Schubert polynomials and key polynomials previously given by Fan and Guo~\cite{fan2022upper}, but also implies the Lorentzian property for those polynomials that attain their upper bounds~\cite[Remark 1.2]{PengLinSun2024}. This greatly advances the resolution of the Lorentzian conjecture proposed by Huh et al.~\cite{Huh} regarding the normalized dual characters of flagged Weyl modules; see~\cite{BH} for more background on Lorentzian polynomials and \cite{fomin,GL} for more information of flagged Weyl module.   
    
    In this paper, we study the characterization of the bounds of Schur characters. These characters are known to contain more monomials than the characters of flagged Weyl modules, see~\cite{WYZZZ}. Magyar~\cite{Magyar-1998-2} established a Weyl-type character formula, which is the explicit formula for Schur characters in the case of general linear groups.
    As pointed out in \cite[Chapter 6]{Fulton}, skew Schur polynomials can be realized as special cases of dual characters of Weyl modules.
    Consequently, the characterization of the bounds of Schur characters may be applied to the Lorentzian conjecture for skew Schur polynomials proposed by Huh et al.~\cite{Huh}.    
    Another important class of Schur characters arises from Rothe diagrams. For a permutation $w$, the Schur function associated with its Rothe diagram agrees with the Stanley symmetric function $F_w$. Stanley symmetric functions were introduced in connection with reduced decompositions of permutations and form a distinguished family of Schur-positive symmetric functions; in particular, $F_w$ is a single Schur function precisely when $w$ is vexillary~\cite{Stanley1984,LascouxSchutzenberger1985}.

    Let $C=(C_1,C_2,\ldots,C_n)$ and $D=(D_1,D_2,\ldots,D_n)$ be two diagrams contained in $[n]^2$. We establish the explicit criteria for the Schur character $\chi_D$ to attain its coefficient-wise bounds. Let ${\bf x}=(x_1,\ldots,x_n)$ and for a diagram $C=(C_1,\ldots,C_n)$, define ${\bf x}^C=\prod_{j=1}^n\prod_{i\in C_j}x_i$. Then the bounds can be stated as follows:
	\begin{equation}\label{eq-ineq-x}
	\sum_{{\bf x}^a\in \{{\bf x}^C:\, |C_j|=|D_j|,\,\forall\, j\in[n] \}}{\bf x}^a\le \chi_D({\bf x})\le \sum_{\{C:\, |C_j|=|D_j|,\, \forall \, j\in [n]\}}{\bf x}^C,	
	\end{equation}
	where for two polynomials $f(\mathbf{x})=\sum_\alpha a_\alpha{\bf x}^\alpha$ and $g(\mathbf{x})=\sum_\alpha b_\alpha{\bf x}^\alpha$, we say $f(\mathbf{x})\le g(\mathbf{x})$ if $a_\alpha\le b_\alpha$ for all $\alpha$. Applying the principal specialization, the above inequality is equivalent to 
	\begin{equation}\label{eq-ineq-1}
		|\{{\bf x}^C:\, |C_j|=|D_j|,\,\forall\, j\in [n] \}|\le \chi_D(1,\ldots,1)\le |\{C:\,|C_j|=|D_j|, \,\forall \, j\in [n] \}|.
	\end{equation}
	
	To present our main results, we also need to introduce several definitions. Let $D$ be a diagram contained in $[n]^2$. We refer to each square $(i,j)\in D$ as a box, and call the remaining squares in $[n]^2\setminus D$ empty grids. 
    Since empty columns contribute only the factor $1$ to $\mathbf{x}^C$ and full columns contribute a fixed monomial $x_1x_2\cdots x_n$, they do not affect whether Schur characters attain the lower bounds. Therefore we may remove all empty columns and all full columns from a diagram $D$, and define the remaining part as the valid part of $D$. 
    The first main result of this paper is stated as follows.
	\begin{thm}\label{thm-lower}
		The Schur character $\chi_D({\bf x})$ is zero-one (attains the lower bound of \eqref{eq-ineq-x}) if and only if the valid part of $D$ satisfies one of the following conditions:		
		\begin{enumerate}
			\item[(a)] all boxes are in the same column;
			\item[(b)] all boxes are in the same row;
			\item[(c)] all empty grids 
            lie in a single row.
		\end{enumerate}
	\end{thm}

    For Schur characters to attain the upper bounds, they must avoid not only the subdiagram required for the characters of flagged Weyl modules (as shown on the left of Figure~\ref{fig-upper}), but also the variation caused by row permutations.
	\begin{thm}\label{thm-upper}
		The Schur character $\chi_D({\bf x})$ attains the upper bound of \eqref{eq-ineq-x} if and only if $D$ avoids the configurations depicted in Figure~\ref{fig-upper}.
	\end{thm}
	
		\begin{figure}[ht]
		\centering
		\begin{minipage}{0.35\textwidth}
			\begin{tikzpicture}
				\centering
				\draw[dashed, black] (-1.75,2) -- (1.75,2);
				\draw[dashed, black] (-1.75,1.5) -- (1.75,1.5);
				\draw[dashed, black] (-1.75,0) -- (1.75,0);
				\draw[dashed, black] (-1.75,0.5) -- (1.75,0.5);
				\draw[dashed, black] (-0.5,-0.5) -- (-0.5,2.5);
				\draw[dashed, black] (-1,-0.5) -- (-1,2.5);
				\draw[dashed, black] (1,-0.5) -- (1,2.5);
				\draw[dashed, black] (0.5,-0.5) -- (0.5,2.5);
				\draw[black] (-1,1.5) rectangle (-0.5,2);
				\draw[black] (0.5,1.5) rectangle (1,2);
				\draw[black] (-1,1.5) -- (-0.5,2);
				\draw[black] (0.5,1.5) -- (1,2);
				\draw[black] (-1,2) -- (-0.5,1.5);
				\draw[black] (0.5,2) -- (1,1.5);
				\node at (-2.25,1.75) {$i_1$};
				\node at (-2.25,0.25) {$i_2$};
				\node at (-0.75,2.85) {$j_1$};
				\node at (0.75,2.85) {$j_2$};
				\filldraw [lightgray] (-1,0) rectangle (-0.5,0.5);
				\draw[black] (-1,0) rectangle (-0.5,0.5);
				\filldraw [lightgray] (0.5,0) rectangle (1,0.5);
				\draw[black] (0.5,0) rectangle (1,0.5);
			\end{tikzpicture}
		\end{minipage}
		\begin{minipage}{0.35\textwidth}
			\begin{tikzpicture}
				\centering
				\draw[dashed, black] (-1.75,2) -- (1.75,2);
				\draw[dashed, black] (-1.75,1.5) -- (1.75,1.5);
				\draw[dashed, black] (-1.75,0) -- (1.75,0);
				\draw[dashed, black] (-1.75,0.5) -- (1.75,0.5);
				\draw[dashed, black] (-0.5,-0.5) -- (-0.5,2.5);
				\draw[dashed, black] (-1,-0.5) -- (-1,2.5);
				\draw[dashed, black] (1,-0.5) -- (1,2.5);
				\draw[dashed, black] (0.5,-0.5) -- (0.5,2.5);
				\draw[black] (-1,0) rectangle (-0.5,0.5);
				\draw[black] (0.5,0) rectangle (1,0.5);
				\draw[black] (-1,0) -- (-0.5,0.5);
				\draw[black] (0.5,0) -- (1,0.5);
				\draw[black] (-1,0.5) -- (-0.5,0);
				\draw[black] (0.5,0.5) -- (1,0);
				\node at (-2.25,1.75) {$i_1$};
				\node at (-2.25,0.25) {$i_2$};
				\node at (-0.75,2.85) {$j_1$};
				\node at (0.75,2.85) {$j_2$};
				\filldraw [lightgray] (-1,1.5) rectangle (-0.5,2);
				\draw[black] (-1,1.5) rectangle (-0.5,2);
				\filldraw [lightgray] (0.5,1.5) rectangle (1,2);
				\draw[black] (0.5,1.5) rectangle (1,2);
			\end{tikzpicture}
		\end{minipage}
		\caption{Two forbidden configurations for $D$.}
		\label{fig-upper}
	\end{figure}
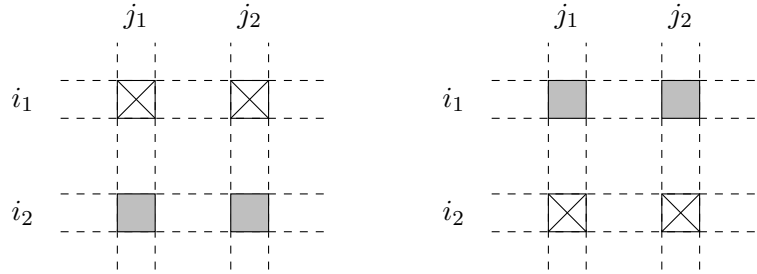
	
	This paper is organized as follows. In Section~\ref{sec-pre}, we present the basic definitions and notation. Sections~\ref{sec-lower} and~\ref{sec-upper} are devoted to the proofs of Theorems~\ref{thm-lower} and~\ref{thm-upper}, respectively. At the end of each section, we specialize the corresponding result to Rothe diagrams and obtain a pattern-avoidance characterization for Stanley symmetric functions.

	\section{Preliminaries}\label{sec-pre}
    In this section, we recall the basic definitions and background related to Schur characters, mainly following the construction of Magyar \cite{Magyar-1998-2}; see also~\cite[Chapter 6]{Fulton}.
	
	For any diagram $D$, let  $\Sigma_D$ be the symmetric group permuting the elements of $D$, and we define two subgroups of the symmetric group $\Sigma_D$ by
	\begin{equation*}
		{\rm Col}(D)=\{\pi  \in \Sigma_D \mid \mbox{ for any $(i,j)\in D$  there exists $i'$ such that $\pi(i,j)=(i^{\prime},j)$}\},
	\end{equation*}
	\begin{equation*}
		{\rm Row}(D)=\{\pi  \in \Sigma_D \mid  \mbox{ for any $(i,j)\in D$  there exists $j'$ such that $\pi(i,j)=(i,j')$}\}.
	\end{equation*}
	Set
	\begin{equation*}
		\alpha_{D}=\frac{1}{|{\rm Row}(D)|}\sum_{\pi \in {\rm Row}(D)}\pi \quad\text{ and }\quad \beta_{D}=\frac{1}{|{\rm Col}(D)|}\sum_{\pi \in {\rm Col}(D)}{\rm sgn}(\pi)\pi,
	\end{equation*}
	where ${\rm sgn}(\pi)$ is the sign of the permutation $\pi$.
	Then the Schur module is defined by
	\begin{equation*}
		\mathbb{S}_{D}={\rm Im}\left(\alpha_D \beta_D\mid_{V^{\otimes |D|}}  \right)=V^{\otimes |D|} \alpha_D \beta_D \subset V^{\otimes |D|},
	\end{equation*}
	which is a $G$-module under the diagonal action of the general linear group $G$.
	
	As aforementioned, Weyl modules are the dual of Schur modules as $G$-modules \cite{Magyar-1998-2}. In this paper, we adopt the following definition of Weyl modules in terms of determinants~\cite{magyar1998schubert}.
	Let $Y=\left( y_{i,j}\right)_{1\le i, j\le n}$ be an $n\times n$ matrix with indeterminates $\{y_{i,j}\}_{1\le i, j\le n}$, and
	let $\mathbb{C}[Y]$ be the polynomial ring in the indeterminates $\{y_{i,j} \}_{1\le i, j \le n}$.
	For any diagram $D=\left(D_1,\ldots,D_n \right)$, the Weyl module $\mathbb{W}_{D}$ is defined by
	\begin{equation}\label{def-Weyl-module-matrix}
		\mathbb{W}_{D}={\rm Span}_{\mathbb{C}}\left\{\prod_{j=1}^{n}{\rm det}\left(Y_{D_j}^{C_j} \right) \mid C=\left( C_1,\ldots,C_n\right), C_j\subseteq [n] \text{ and } |C_j|=|D_j|\text{ for }j\in [n],  \right\},
	\end{equation}
	where $Y_{D_j}^{C_j}$ denotes the submatrix of $Y$ obtained by restricting to rows $C_j$ and columns $D_j$.
	The Weyl module $\mathbb{W}_{D}$
	is also a $G$-module under the $G$-action on $\mathbb{C}[Y]$ defined by
	\begin{equation*}
		f(Y)\cdot g=f\left(g^{-1}Y \right), \forall g\in G.
	\end{equation*}
	Magyar \cite{Magyar-1998-2} established the duality between $\mathbb{S}_D$ and $\mathbb{W}_D$ as follows.
	
	\begin{prop}[{\cite[Proposition 1]{Magyar-1998-2}}]
		\label{prop-dual}
		For any diagram $D$, the dual of
		$\mathbb{S}_D$ is isomorphic to $\mathbb{W}_D$ as $G$-modules.
	\end{prop}
	
	
	We are mainly concerned with the characters of $\mathbb{S}_D$ and dual characters of $\mathbb{W}_D$. Let $N$ be a $G$-module or $B$-module over $\mathbb{C}$, and let $X$ be the diagonal matrix with entries $x_1,x_2,\ldots,x_n$. By taking $N$ as a vector space over $\mathbb{C}$,
	the action of $X$ on $N$ can be seen as a linear map  $X:\, N\to N$. Then the character of $N$ is given by
	\begin{equation*}
		{\rm char}(N)(x_1,\ldots ,x_n)={\rm tr}(X:\, N\to N),
	\end{equation*}
	and the dual character of $N$ is defined as the character of $N^*$:
	\begin{align*}
		{\rm char}^*(N)(x_1,\ldots,x_n)&={\rm tr}(X:\, N^*\to N^*)
		={\rm char}(N)(x_1^{-1},\ldots,x_n^{-1}).
	\end{align*}

	In this paper, for a diagram $D$, we refer to the dual character of the Weyl module $\mathbb{W}_D$ as the Schur character, and denote it by
	\[
	\chi_D := \operatorname{char}^*(\mathbb{W}_D).
	\]
    Equivalently, $\chi_D$ is also the character of the corresponding Schur module $\mathbb{S}_{D}$ by Proposition~\ref{prop-dual}. 
	Given two diagrams $D, D'$, we say they are equivalent, denoted $D \simeq D'$, if $D$ can be obtained from $D'$ by permuting rows and columns. By the definition of the Weyl module \eqref{def-Weyl-module-matrix} and the duality between Schur and Weyl modules (Proposition~\ref{prop-dual}), permuting rows or columns of a diagram $D$ yields an isomorphic Schur module $\mathbb{S}_D$ (see also~\cite[Page 16]{Liu2010}). Consequently, for equivalent diagrams $D \simeq D'$, the corresponding Schur characters coincide: $\chi_D = \chi_{D'}$.
    
    In particular, consider the case where $D$ is the Young diagram $D(\lambda/\mu)$ associated with the skew partition $\lambda/\mu$. Then the Schur character $\chi_{D(\lambda/\mu)}$ is closely related to the skew Schur polynomial. To make this precise, we first recall some basic definitions.
    A partition of $d$ is a weakly decreasing sequence $\lambda = (\lambda_1, \lambda_2, \dots, \lambda_\ell)$ such that $\lambda_1 + \lambda_2 + \cdots + \lambda_\ell = d$. Each partition $\lambda$ can be represented by its Young diagram $D(\lambda)$, which consists of $\ell(\lambda)$ rows of left-justified squares, with the $i$-th row containing $\lambda_i$ squares for $1 \leq i \leq \ell(\lambda)$. Given two partitions $\lambda$ and $\mu$ satisfying $\mu_i \leq \lambda_i$ for all $i$, we define the skew partition $\lambda/\mu$. Its Young diagram $D(\lambda/\mu)$ is obtained by removing the Young diagram of $\mu$ from the top-left corner of that of $\lambda$. A semistandard Young tableau (SSYT) of shape $\lambda/\mu$ is a filling $T = (T_{ij})$ of the squares of $D(\lambda/\mu)$ with positive integers such that entries are weakly increasing along each row and strictly increasing down each column. If, for each $i \geq 1$, the entry $i$ appears exactly $\alpha_i$ times in $T$, we call $\alpha = (\alpha_1, \alpha_2, \dots)$ the type of $T$, denoted $\mathrm{type}(T)$.
    Given a skew partition $\lambda/\mu$, the skew Schur polynomial in $n$ variables $\mathbf{x} = (x_1, \dots, x_n)$ is defined by
    \begin{align}\label{eq-def-Schur}
        s_{\lambda/\mu}(\mathbf{x}) = \sum_{T}  \mathbf{x}^{T},
    \end{align}
    where the sum ranges over all SSYT of shape $\lambda/\mu$ and for such a tableau $T$ with content $\alpha = (\alpha_1,\alpha_2,\ldots)$ (where $\alpha_i$ is the number of entries equal to $i$ in $T$), we set $\mathbf{x}^{T} = \prod_{i \geq 1} x_i^{\alpha_i}$.
    It is well known that the skew Schur polynomial in $n$ variables coincides with the Schur character $\chi_{D(\lambda/\mu)}(x_1, \dots, x_n)$, where $D(\lambda/\mu)$ is viewed as a Young diagram embedded in an $n \times n$ grid of squares.
	\begin{thm}[{\cite[Chapter 6]{Fulton}}]\label{them-Schur-Schur module}
		Given a skew partition $\lambda/\mu$, we have
		\begin{align}\label{eq-skew-schur-dual-weyl}
			\chi_{D(\lambda/\mu)}(x_1,\ldots,x_n)=s_{\lambda/\mu}(x_1,\ldots,x_n).
		\end{align}
	\end{thm}

	
	By a similar consideration as stated in \cite[{p. 471}]{2018Schubert}, Wang et al.~\cite{WYZZZ} gave the following characterization of monomials in Schur characters.
	
	\begin{prop}[{\cite[Proposition 4.4]{WYZZZ}}]\label{prop-monomial-dual-Weyl}
		For any diagram $D\subseteq [n]^2$, the monomials in the Schur character $\chi_D(x_1,\ldots,x_n)$ are precisely those of the form
		\[\left\{\prod_{j=1}^n\prod_{i\in C_j}x_i \mid C_j\subseteq [n] \text{ and }|C_j|=|D_j|\text{ for each }j\in [n] \right\}.\]
	\end{prop}
	
Finally, we consider another special diagram. Let $w=w_1\cdots w_n\in \mathfrak{S}_n$. An inversion of $w$ is a pair $(i,k)$ such that $i<k$ and $w_i>w_k$. We say that $w$ contains a pattern $v=v_1\cdots v_m\in\mathfrak{S}_m$ if there exist indices $1\leq i_1<\cdots<i_m\leq n$ such that $w{i_1},\ldots,w{i_m}$ has the same relative order as $v_1,\ldots,v_m$; otherwise, $w$ is said to avoid $v$. The Rothe diagram of $w$ is
\[
D(w)=\{(i,j)\in[n]^2 \mid j<w_i,\ i<w^{-1}(j)\}.
\]
Recall that inversions of $w$ are in bijection with the boxes of $D(w)$ via
\[
(i,k)\longmapsto (i,w_k),\text{ for }i<k\text{ and } w_i>w_k.
\]
Under the generalized Schur function convention for arbitrary diagrams, the Schur function of $D(w)$ agrees with the Stanley symmetric function in the sense that $s_{D(w)}=F_w$.
Thus results concerning Schur modules of Rothe diagrams may be interpreted directly as results concerning Stanley symmetric functions.

In general, $F_w$ need not be a single Schur function. A permutation $w$ is called vexillary if it avoids the pattern $2143$. The following classical result characterizes precisely when a Stanley symmetric function is a single Schur function.
\begin{thm}[{\cite{Stanley1984,LascouxSchutzenberger1985}}]
    Let $w\in S_n$. Then the Stanley symmetric function $F_w$ is a single Schur function if and only if $w$ is vexillary. More precisely,
    \[ F_w=s_{\lambda(w)},\]
where $\lambda(w)$ is the partition obtained by rearranging the row lengths of the Rothe diagram $D(w)$ in weakly decreasing order.
\end{thm}
More generally, if the Rothe diagram of $w$ is equivalent, up to permutations of rows and columns, to a skew Young diagram $\lambda/\mu$, then we can obtain that $F_w=s_{\lambda/\mu}$.
Hence Stanley symmetric functions extend both ordinary and skew Schur functions.
	
	\section{Proof of Theorem~\ref{thm-lower}}\label{sec-lower}
	
	This section is devoted to the proof of Theorem~\ref{thm-lower}. Following the construction in~\cite{PengLinSun2024}, for diagrams $D=(D_1,D_2,\ldots,D_n)$ and $C=(C_1,C_2,\ldots,C_n)$ of $[n]^2$ with $|C_j|=|D_j|$ for each $j\in[n]$, we define the filling $F$ of $D$ with respect to $C$ by filling the boxes in the $j$-th column of $D$ with the elements of $C_j$, and denote the entry filled in the box $(i,j)$ by $c_{ij}$.
    Let $\mathcal{F}_D(C)$ denote the set of those fillings with respect to diagrams $C$ and $D$.
    Following the definitions in~\cite{PengLinSun2024},  for any filling $F\in\mathcal{F}_D(C)$, the inversion 
    set of $F$ is 
    \begin{equation}\label{eq-inv}
        {\rm inv} (F)=\left\{ \left((i,j),\,(k,j)\right) \mid i<k,\,1\le j\le n \text{ and } c_{ij}>c_{kj} \text{ for any } i,k\in D_j\right\},
    \end{equation}
    and the sign of $F$ is ${\rm sgn}(F)=(-1)^{|{\rm inv}(F)|}$.
    Then each monomial of the polynomial $\det(Y^C_D):=\prod_{j=1}^n \det(Y^{C_j}_{D_j})$
    can be represented by some filling $F \in \mathcal{F}_D(C)$ as follows.
	\begin{lem}\label{lem-det-y^F}
	Given two diagrams $D=(D_1,D_2,\ldots,D_n)$ and $C=(C_1,C_2,\ldots,C_n)$ of $[n]^2$ with $|C_j|=|D_j|$ for each $j\in[n]$, we have
    \begin{equation}\label{eq-filling-det}
        \det(Y^C_D)=\sum_{F\in \mathcal{F}_D(C) }{\rm sgn}(F) y^{F},
    \end{equation}
    where $y^F=\prod_{(i,j)\in D}y_{c_{ij},i}$.
	\end{lem}
    \begin{proof}
    The proof follows the argument similar to  that of \cite[Lemma~2.2]{PengLinSun2024}.
    \end{proof}

    Now we are in a position to present the first main result of this section.
    \begin{thm}\label{thm-lower-equivalent-condition}
        The diagram $D=(D_1,\ldots, D_n)$ satisfies one of the conditions illustrated in Theorem~\ref{thm-lower}, if and only if the equality $\det\left(Y^C_D\right)=\det\left(Y^{C'}_D\right)$ holds for each pair $(C,C')$ satisfying
		\begin{equation}\label{eq-lower-condition}
			|C_j|=|C'_j|=|D_j| \text{ for }j\in[n] \text{ and } {\bf x}^C={\bf x}^{C'}.
		\end{equation}
    \end{thm}
    \begin{proof}
    When no ambiguity arises, we may identify $D$ with its valid part, and the same holds for any diagram $C$ with $|C_j| = |D_j|$ for all  $j \in [n]$. In this sense, we suppose $D$ contains $m$ columns.
    
        We first prove the sufficiency. 

        \emph{Case (a): all boxes are in the same column.}

        The diagram $D$ consists of a single column.
        Then, when there is no ambiguity, we abuse notation simply by writing $|C|=|D|$ since each consists of a single set in this case. The monomial $\textbf{x}^C$ is uniquely determined by $C$.
        Hence $\textbf{x}^C=\textbf{x}^{C'}$ implies $C=C'$, and the equality of determinants follows immediately.
        
        \emph{Case (b): all boxes are in the same row.}

        For $1\le j \le m$, the $j$-th column of $D$ contains exactly one box.
        Hence for any $C$, each filling $F \in \mathcal{F}_D(C)$ is uniquely determined by $C$. 
        Thus $\mathcal{F}_D(C)$ and $\mathcal{F}_D(C')$ each contain exactly one filling, say $F$ and $F'$ respectively.
        Since $\textbf{x}^C=\textbf{x}^{C'}$, the multisets of entries in $C$ and $C'$ are identical. 
        Therefore, $F'$ can be obtained from $F$ by interchanging the entries of $F$ which implies that $y^{F}=y^{F'}$. Thus we obtain that $|\det(Y^C_D)| = |\det(Y^{C'}_D)|$. Since ${\rm sgn}(F)$ is invariant under the interchanging of the entries of $F$, we have $\det(Y^C_D) = \det(Y^{C'}_D)$ by Lemma~\ref{lem-det-y^F}.
        
        \emph{Case (c): all empty grids lie in a single row.}
        
        In this case, since all empty cells lie in a single row (say row $r$), each column of $D$ contains exactly one empty cell, and that empty cell is in row $r$. Hence for each column $j$, we have $D_j = [n] \setminus \{r\}$. For any $C$ with $|C_j| = |D_j| = n-1$, we can write $C_j = [n] \setminus \{\ell_j\}$ for some $\ell_j \in [n]$. Similarly, for $C'$ we write $C'_j = [n] \setminus \{\ell'_j\}$. The monomials are then
        $${\bf x}^C=\frac{x_1^mx_2^m\cdots x_n^m}{x_{\ell_1}x_{\ell_2}\cdots x_{\ell_m}}\qquad \text{and}\qquad {\bf x}^{C'}=\frac{x_1^mx_2^m\cdots x_n^m}{x_{\ell'_1}x_{\ell'_2}\cdots x_{\ell'_m}}.$$
        Since ${\bf x}^C={\bf x}^{C'}$, we have the multisets $\{\ell_1,\ldots,\ell_m\}=\{\ell'_1,\ldots,\ell'_m\}$. 
        So we can establish a bijection between $\mathcal{F}_D(C)$ and $\mathcal{F}_D(C')$ by interchanging the columns of each filling, that is, columns with the same label are moved to the resulting positions with respect to $C'$. Since interchanging the columns of $F$ does not change the sign of a filling, we have ${\rm sgn}(F) = {\rm sgn}(F')$ for corresponding fillings. Consequently, $\det(Y^C_D) = \det(Y^{C'}_D)$ by Lemma~\ref{lem-det-y^F}.
        For example, for diagrams $$C=\{\{3\},\{2\},\{1\},\{1\},\{2\},\{4\},\{2\}\},\qquad C'=\{\{2\},\{3\},\{1\},\{4\},\{1\},\{2\},\{2\}\},$$ this bijection is depicted in Figure~\ref{fig-ex-Case(c)}.

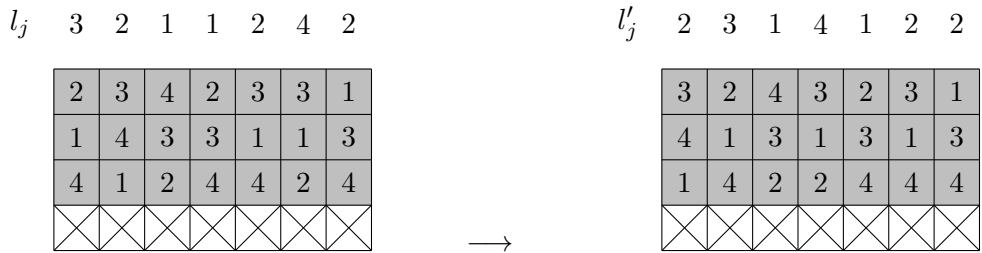
\begin{figure}[h]
    \centering
    \begin{minipage}[b]{0.45\textwidth}
        \centering
        \begin{tikzpicture}[scale=0.6]

            \filldraw [lightgray] (0,0) rectangle (7,3);
            \foreach \x in {0,1,2,3,4,5,6,7} {
                \draw[black] (\x, -1) -- (\x, 3);
            }
            \foreach \y in {0,1,2,3} {
                \draw[black] (0, \y) -- (7, \y);
            }
            \foreach \x in {0,1,2,3,4,5,6} {
                \draw[black] (\x, -1) -- (\x+1, 0);
            }
            \foreach \x in {0,1,2,3,4,5,6} {
                \draw[black] (\x, 0) -- (\x+1, -1);
            }
            \draw[black] (0, -1) -- (7, -1);
            
            \node at (-0.75, 4) {$l_j$};
            
            \node at (0.5, 0.5) {4};
            \node at (1.5, 0.5) {1};
            \node at (2.5, 0.5) {2};
            \node at (3.5, 0.5) {4};
            \node at (4.5, 0.5) {4};
            \node at (5.5, 0.5) {2};
            \node at (6.5, 0.5) {4};
            
            \node at (0.5, 1.5) {1};
            \node at (1.5, 1.5) {4};
            \node at (2.5, 1.5) {3};
            \node at (3.5, 1.5) {3};
            \node at (4.5, 1.5) {1};
            \node at (5.5, 1.5) {1};
            \node at (6.5, 1.5) {3};
            
            \node at (0.5, 2.5) {2};
            \node at (1.5, 2.5) {3};
            \node at (2.5, 2.5) {4};
            \node at (3.5, 2.5) {2};
            \node at (4.5, 2.5) {3};
            \node at (5.5, 2.5) {3};
            \node at (6.5, 2.5) {1};
            
            \node at (0.5, 4) {3};
            \node at (1.5, 4) {2};
            \node at (2.5, 4) {1};
            \node at (3.5, 4) {1};
            \node at (4.5, 4) {2};
            \node at (5.5, 4) {4};
            \node at (6.5, 4) {2};
        \end{tikzpicture}
    \end{minipage}
     $\longrightarrow$
    \begin{minipage}[b]{0.45\textwidth}
        \centering
        \begin{tikzpicture}[scale=0.6]
             \filldraw [lightgray] (0,0) rectangle (7,3);
            \foreach \x in {0,1,2,3,4,5,6,7} {
                \draw[black] (\x, -1) -- (\x, 3);
            }
            \foreach \y in {0,1,2,3} {
                \draw[black] (0, \y) -- (7, \y);
            }
            \draw[black] (0, -1) -- (7, -1);
            
            \foreach \x in {0,1,2,3,4,5,6} {
                \draw[black] (\x, -1) -- (\x+1, 0);
            }
            \foreach \x in {0,1,2,3,4,5,6} {
                \draw[black] (\x, 0) -- (\x+1, -1);
            }
            \node at (-0.75, 4) {$l'_j$};
            
            \node at (0.5, 0.5) {1};
            \node at (1.5, 0.5) {4};
            \node at (2.5, 0.5) {2};
            \node at (3.5, 0.5) {2};
            \node at (4.5, 0.5) {4};
            \node at (5.5, 0.5) {4};
            \node at (6.5, 0.5) {4};
            
            \node at (0.5, 1.5) {4};
            \node at (1.5, 1.5) {1};
            \node at (2.5, 1.5) {3};
            \node at (3.5, 1.5) {1};
            \node at (4.5, 1.5) {3};
            \node at (5.5, 1.5) {1};
            \node at (6.5, 1.5) {3};
            
            \node at (0.5, 2.5) {3};
            \node at (1.5, 2.5) {2};
            \node at (2.5, 2.5) {4};
            \node at (3.5, 2.5) {3};
            \node at (4.5, 2.5) {2};
            \node at (5.5, 2.5) {3};
            \node at (6.5, 2.5) {1};
            
            \node at (0.5, 4) {2};
            \node at (1.5, 4) {3};
            \node at (2.5, 4) {1};
            \node at (3.5, 4) {4};
            \node at (4.5, 4) {1};
            \node at (5.5, 4) {2};
            \node at (6.5, 4) {2};
        \end{tikzpicture}
    \end{minipage}
    \caption{An instance of the bijection in Case (c).}
    \label{fig-ex-Case(c)}
\end{figure}


        {For the necessity, suppose that $D$ satisfies none of conditions \textup{(a)}, \textup{(b)}, and \textup{(c)}. We identify $D$ with its valid part, so every column is a nonempty proper subset of $[n]$, and $D$ has at least two columns. We claim that, up to row and column permutations, $D$ contains one of Configurations \textup{(I)}, \textup{(II)}, and \textup{(III)} in Figure~\ref{fig-upper1}.}
        
         {If two columns $D_{j_1}$ and $D_{j_2}$ are distinct and incomparable under inclusion, choose $i_1\in D_{j_2}\setminus D_{j_1}$ and $i_2\in D_{j_1}\setminus D_{j_2}$. Then the rows $i_1,i_2$ and columns $j_1,j_2$ form Configuration \textup{(II)}. If instead the two columns are comparable, we may assume that $D_{j_2}\subsetneq D_{j_1}$. Choose $i_1\in D_{j_2}$ and $i_2\in D_{j_1}\setminus D_{j_2}$. Then these rows and columns form Configuration \textup{(I)}. It remains to consider the case in which all columns are equal, say $D_j=S$ for every $j$. Since conditions \textup{(b)} and \textup{(c)} both fail, we have
        $$|S|\geq 2\qquad\text{and}\qquad|[n]\setminus S|\geq 2.$$
        Choose distinct rows $i_1,i_2\in S$ and $i_3,i_4\in[n]\setminus S$, together with any two distinct columns $j_1,j_2$. These rows and columns form Configuration \textup{(III)}.
        Thus, whenever $D$ satisfies none of conditions \textup{(a)}, \textup{(b)}, and \textup{(c)}, it contains, up to row and column permutations, at least one of the three configurations in Figure~\ref{fig-upper1}.}
        
        
        Suppose the chosen configuration occurs in rows $i_1, i_2$ (and possibly $i_3, i_4$ for (III)) and columns $j_1, j_2$. 
        Note that we do not consider the configuration \pattern{scale=1.2}{2}{0/0,0/1,1/1}{1/0}, since $D$ contains no empty columns. More precisely, if $D$ contains at least one box $(i,j_1)$ with $i < i_1$, then after suitable row and column permutations, this leads to Configuration (I) if $(i,j_2) \in D$ and Configuration (II) if $(i,j_2) \notin D$.

\begin{figure}[ht]
    \centering
    \begin{minipage}[b]{0.3\textwidth}
        \centering
        \begin{tikzpicture}
            \draw[dashed, black] (-0.5,1) -- (1.5,1);
            \draw[dashed, black] (-0.5,0.5) -- (1.5,0.5);
            \draw[dashed, black] (-0.5,0) -- (1.5,0);
            \draw[dashed, black] (0,-0.5) -- (0,1.5);
            \draw[dashed, black] (0.5,-0.5) -- (0.5,1.5);
            \draw[dashed, black] (1,-0.5) -- (1,1.5);
            \filldraw [lightgray] (0,0) rectangle (0.5,0.5);
            \filldraw [lightgray] (0,0.5) rectangle (0.5,1);
            \filldraw [lightgray] (0.5,0.5) rectangle (1,1);
            \draw[black] (0,0) rectangle (0.5,0.5);
            \draw[black] (0.5,0) rectangle (1,0.5);
            \draw[black] (0,0.5) rectangle (0.5,1);
            \draw[black] (0.5,0.5) rectangle (1,1);
            \draw[black] (0.5,0.5) -- (1,0);
            \draw[black] (0.5,0) -- (1,0.5);
            \node at (0.5,-2) {(I)};
        \end{tikzpicture}
    \end{minipage}
    \begin{minipage}[b]{0.3\textwidth}
        \centering
        \begin{tikzpicture}
            \draw[dashed, black] (-0.5,1) -- (1.5,1);
            \draw[dashed, black] (-0.5,0.5) -- (1.5,0.5);
            \draw[dashed, black] (-0.5,0) -- (1.5,0);
            \draw[dashed, black] (0,-0.5) -- (0,1.5);
            \draw[dashed, black] (0.5,-0.5) -- (0.5,1.5);
            \draw[dashed, black] (1,-0.5) -- (1,1.5);
            \filldraw [lightgray] (0,0) rectangle (0.5,0.5);
            \filldraw [lightgray] (0.5,0.5) rectangle (1,1);
            \draw[black] (0,0) rectangle (0.5,0.5);
            \draw[black] (0.5,0) rectangle (1,0.5);
            \draw[black] (0,0.5) rectangle (0.5,1);
            \draw[black] (0.5,0.5) rectangle (1,1);
            \draw[black] (0.5,0.5) -- (1,0);
            \draw[black] (0.5,0) -- (1,0.5);
            \draw[black] (0,0.5) -- (0.5,1);
            \draw[black] (0,1) -- (0.5,0.5);
            \node at (0.5,-2) {(II)};
        \end{tikzpicture}
    \end{minipage}
    \begin{minipage}[b]{0.3\textwidth}
        \centering
        \begin{tikzpicture}
            \draw[dashed, black] (-0.5,1) -- (1.5,1);
            \draw[dashed, black] (-0.5,0.5) -- (1.5,0.5);
            \draw[dashed, black] (-0.5,0) -- (1.5,0);
            \draw[dashed, black] (-0.5,1.5) -- (1.5,1.5);
            \draw[dashed, black] (-0.5,-0.5) -- (1.5,-0.5);
            \draw[dashed, black] (0,-1) -- (0,2);
            \draw[dashed, black] (0.5,-1) -- (0.5,2);
            \draw[dashed, black] (1,-1) -- (1,2);
            \filldraw [lightgray] (0,0.5) rectangle (1,1.5);
            \draw[black] (0,0) rectangle (0.5,0.5);
            \draw[black] (0.5,0) rectangle (1,0.5);
            \draw[black] (0,0.5) rectangle (0.5,1);
            \draw[black] (0.5,0.5) rectangle (1,1);
            \draw[black] (0,0.5) rectangle (0.5,1);
            \draw[black] (0,1) rectangle (0.5,1.5);
            \draw[black] (0.5,1) rectangle (1,1.5);
            \draw[black] (0.5,-0.5) rectangle (1,0.5);
            \draw[black] (0,-0.5) rectangle (0.5,0);
            \draw[black] (0.5,0.5) -- (1,0);
            \draw[black] (0,-0.5) -- (0.5,-0.5);
            \draw[black] (0,0) -- (0.5,-0.5);
            \draw[black] (0,-0.5) -- (0.5,0);
            \draw[black] (0.5,-0.5) -- (1,0);
            \draw[black] (0.5,0) -- (1,-0.5);
            \draw[black] (0,0) -- (0.5,0.5);
            \draw[black] (0,0.5) -- (0.5,0);
            \draw[black] (0.5,0) -- (1,0.5);
            
            \node at (0.5,-1.5) {(III)};
        \end{tikzpicture}
    \end{minipage}
    \caption{Three subdiagrams that $D$ should avoid.}
    \label{fig-upper1}
\end{figure}
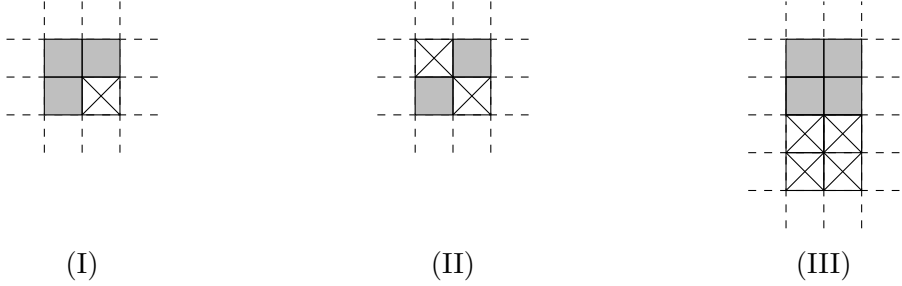
        
			
		Next we want to find a pair $(C,C')$ satisfying the conditions in \eqref{eq-lower-condition} for which $\det\left(Y^C_D\right)\neq \det\left(Y^{C'}_D\right)$. 
        Suppose the chosen configuration occurs in rows $i_1, i_2$ (and possibly $i_3, i_4$ for (III)) and columns $j_1, j_2$. 
		Without loss of generality, we may take $C$ and $C'$ to differ only in columns $j_1$ and $j_2$, and set $C_j = C'_j$ for all $j \notin \{j_1, j_2\}$. 
        Then $\det\left(Y^C_D\right)= \det\left(Y_D^{C'}\right)$ holds if and only if
        \begin{align}\label{subdet}
    \det\left(Y_{D_{j_1}}^{C_{j_1}}\right)\cdot \det\left(Y_{D_{j_2}}^{C_{j_2}}\right)=\det\left(Y_{D_{j_1}}^{C'_{j_1}}\right)\cdot \det\left(Y_{D_{j_2}}^{C'_{j_2}}\right).
    \end{align}

		For configurations (I) and (II), let 
		$$C_{j_1}=\{1,a_1,\ldots,a_{p}\},\qquad C_{j_2}=\{2,b_1,\ldots,b_q\},$$
		$$C'_{j_1}=\{2,a_1,\ldots,a_{p}\},\qquad C'_{j_2}=\{1,b_1,\ldots,b_q\},$$
        where $\{a_1,\ldots,a_{p}\}$ and $\{b_1,\ldots,b_q\}$ are two subsets of $[n]\setminus \{1,2\}$. This construction is well-defined since, by the condition that $D$ contains at least one empty grid in columns $j_1$ and $j_2$, we have $p,\,q \le n-2$.
		We next consider fillings $F \in \mathcal{F}_D(C)$ and $F' \in \mathcal{F}_D(C')$ that coincide everywhere except in the boxes  $(i_1,j_2)$, $(i_2,j_1)$ and $(i_1,j_1)$ (for Configuration (I)). 
        In Configuration (I), the three boxes are filled with $1$, $2$ and $a_1$; in Configuration (II), the two boxes are filled with $1$ and $2$, respectively. See Figure~\ref{fig-ex-AB} for an explicit description.

		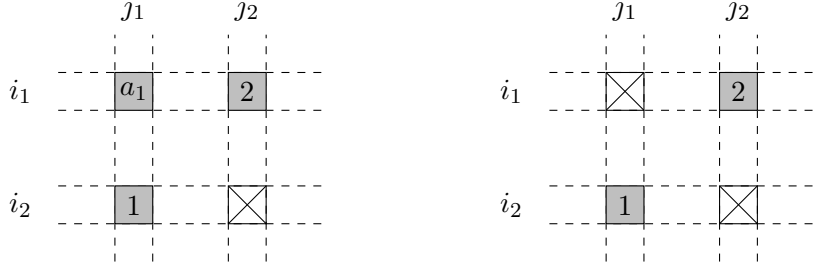
\begin{figure}[h]
        \centering
            \begin{minipage}[b]{0.4\textwidth}
            \centering
            \begin{tikzpicture}
			\draw[dashed, black] (-1.75,2) -- (1.75,2);
			\draw[dashed, black] (-1.75,1.5) -- (1.75,1.5);
			\draw[dashed, black] (-1.75,0) -- (1.75,0);
			\draw[dashed, black] (-1.75,0.5) -- (1.75,0.5);
			\draw[dashed, black] (-0.5,-0.5) -- (-0.5,2.5);
			\draw[dashed, black] (-1,-0.5) -- (-1,2.5);
			\draw[dashed, black] (1,-0.5) -- (1,2.5);
			\draw[dashed, black] (0.5,-0.5) -- (0.5,2.5);
			\draw[black] (0.5,0) rectangle (1,0.5);
			\draw[black] (0.5,0.5) -- (1,0);
			\draw[black] (0.5,0) -- (1,0.5);
            \node at (-2.25,1.75) {$i_1$};
            \node at (-2.25,0.25) {$i_2$};
            \node at (-0.75,2.85) {$j_1$};
            \node at (0.75,2.85) {$j_2$};
			\filldraw [lightgray] (-1,0) rectangle (-0.5,0.5);
			\draw[black] (-1,0) rectangle (-0.5,0.5);
			\filldraw [lightgray] (-1,1.5) rectangle (-0.5,2);
			\draw[black] (-1,1.5) rectangle (-0.5,2);
            \filldraw [lightgray] (0.5,1.5) rectangle (1,2);
            \draw[black] (0.5,1.5) rectangle (1,2);
            \node at (-0.75,1.75) {$a_1$};
            \node at (-0.75,0.25) {1};
            \node at (0.75,1.75) {2};
            \end{tikzpicture}
            \end{minipage}
            \begin{minipage}[b]{0.4\textwidth}
            \centering
            \begin{tikzpicture}
			\draw[dashed, black] (-1.75,2) -- (1.75,2);
			\draw[dashed, black] (-1.75,1.5) -- (1.75,1.5);
			\draw[dashed, black] (-1.75,0) -- (1.75,0);
			\draw[dashed, black] (-1.75,0.5) -- (1.75,0.5);
			\draw[dashed, black] (-0.5,-0.5) -- (-0.5,2.5);
			\draw[dashed, black] (-1,-0.5) -- (-1,2.5);
			\draw[dashed, black] (1,-0.5) -- (1,2.5);
			\draw[dashed, black] (0.5,-0.5) -- (0.5,2.5);
			\draw[black] (0.5,0) rectangle (1,0.5);
			\draw[black] (0.5,0.5) -- (1,0);
			\draw[black] (0.5,0) -- (1,0.5);
            \node at (-2.25,1.75) {$i_1$};
            \node at (-2.25,0.25) {$i_2$};
            \node at (-0.75,2.85) {$j_1$};
            \node at (0.75,2.85) {$j_2$};
			\filldraw [lightgray] (-1,0) rectangle (-0.5,0.5);
			\draw[black] (-1,0) rectangle (-0.5,0.5);
			\draw[black] (-1,1.5) rectangle (-0.5,2);
            \filldraw [lightgray] (0.5,1.5) rectangle (1,2);
            \draw[black] (0.5,1.5) rectangle (1,2);
            \draw[black] (-1,1.5) -- (-0.5,2);
			\draw[black] (-1,2) -- (-0.5,1.5);
            \node at (-0.75,0.25) {1};
            \node at (0.75,1.75) {2};
		\end{tikzpicture}
        \end{minipage}
		    \caption{An instance of fillings in Configurations (I) and (II).}
		    \label{fig-ex-AB}
		\end{figure}
        Then, for Configurations (I) and (II), respectively, consider the coefficient of $y_{1,i_2}$ on both sides of (\ref{subdet}). Denote the corresponding coefficients on the left- and right-hand sides by $L_{y_{1,i_2}}$ and $R_{y_{1,i_2}}$, respectively. We have
        \[
L_{y_{i_1,j_2}}=\pm\det\left(Y^{C_{j_1}\backslash\{1\}}_{D_{j_1}\backslash\{i_2\}}\right)
\det\left(Y^{C_{j_2}}_{D_{j_2}}\right)\neq 0.
\]
        However, on the right-hand side, we have $1\notin C_{j_1}'$, $1\in C_{j_2}'$, and $i_2\notin D'_{j_2}$; hence the variable $y_{1,i_2}$ does not appear on the right-hand side, which implies $R_{y_{i_1,j_2}}=0$.


        For Configuration (III), suppose this structure occurs in rows $i_1,i_2,i_3$ and $i_4$ and in columns $j_1$ and $j_2$. 
		Let
		$$C_{j_1}=\{1,2,a_1,\ldots,a_{p}\},\qquad C_{j_2}=\{3,4,b_1,\ldots,b_q\},$$
		$$C'_{j_1}=\{1,3,a_1,\ldots,a_{p}\},\qquad C'_{j_2}=\{2,4,b_1,\ldots,b_q\},$$
        where $p=\left|D_{j_1}\right|-2,~q=\left|D_{j_2}-2\right|$.
        By a similar argument, consider the coefficient of $y_{1,i_1}y_{3,i_1}$ on both sides of (\ref{subdet}), and denote the corresponding coefficients by $L_{y_{1,i_1}y_{3,i_1}}$ and $R_{y_{1,i_1}y_{3,i_1}}$, respectively. There exist some fillings belonging to  $\mathcal{F}_D(C)$ with rows $i_1$ and $i_2$, as illustrated in Figure~\ref{fig-ex-C}. We have 
        \[ L_{y_{1,i_1}y_{3,i_1}}=\pm\det\left(Y^{C_{j_1}\backslash\{1\}}_{D_{j_1}\backslash\{i_1\}}\right)\cdot\det\left(Y^{C_{j_2}\backslash\{3\}}_{D_{j_2}\backslash\{i_1\}}\right)\neq 0.\]
        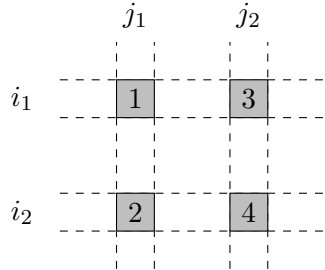
\begin{figure}[ht]
        \centering
        \begin{tikzpicture}
			\centering
			\draw[dashed, black] (-1.75,2) -- (1.75,2);
			\draw[dashed, black] (-1.75,1.5) -- (1.75,1.5);
			\draw[dashed, black] (-1.75,0) -- (1.75,0);
			\draw[dashed, black] (-1.75,0.5) -- (1.75,0.5);
			\draw[dashed, black] (-0.5,-0.5) -- (-0.5,2.5);
			\draw[dashed, black] (-1,-0.5) -- (-1,2.5);
			\draw[dashed, black] (1,-0.5) -- (1,2.5);
			\draw[dashed, black] (0.5,-0.5) -- (0.5,2.5);
            \node at (-2.25,1.75) {$i_1$};
            \node at (-2.25,0.25) {$i_2$};
            \node at (-0.75,2.85) {$j_1$};
            \node at (0.75,2.85) {$j_2$};
			\filldraw [lightgray] (-1,0) rectangle (-0.5,0.5);
			\draw[black] (-1,0) rectangle (-0.5,0.5);
			\filldraw [lightgray] (-1,1.5) rectangle (-0.5,2);
			\draw[black] (-1,1.5) rectangle (-0.5,2);
            \filldraw [lightgray] (0.5,1.5) rectangle (1,2);
            \draw[black] (0.5,1.5) rectangle (1,2);
            \filldraw [lightgray] (0.5,0) rectangle (1,0.5);
            \draw[black] (0.5,0) rectangle (1,0.5);
            \node at (-0.75,1.75) {1};
            \node at (-0.75,0.25) {2};
            \node at (0.75,1.75) {3};
            \node at (0.75,0.25) {4};
		\end{tikzpicture}
		    \caption{An instance of fillings in Configuration (III).}
		    \label{fig-ex-C}
		\end{figure}
        But for $C'$, element $1$ and $3$ are in the same set $C_{j_1}'$, so they cannot occurs in row $i_1$ simultaneously, which implies that $R_{y_{1,i_1}y_{3,i_1}}=0$.

        In all cases, the equality~\eqref{subdet} fails, which completes the proof. 
    \end{proof}

Then we can give the proof of Theorem~\ref{thm-lower} based on Theorem~\ref{thm-lower-equivalent-condition}.
	\begin{proof}[Proof of Theorem~\ref{thm-lower}]

    By Definition~\eqref{def-Weyl-module-matrix} and Proposition~\ref{prop-monomial-dual-Weyl}, the coefficient of a monomial $\mathbf{x}^a$ in $\chi_D(\mathbf{x})$ equals the dimension of
    \begin{equation*} 
        \operatorname{Span}_{\mathbb{C}}\left\{\det\left(Y_{D}^C\right) \;\middle|\;
        C=(C_1,\ldots,C_n),\ C_j\subseteq [n],\ |C_j|=|D_j|\text{ for }j\in [n],\text{ and } \mathbf{x}^C=\mathbf{x}^a\right\}
    \end{equation*} 
    Hence $\chi_D(\mathbf{x})$ is zero-one if and only if every nonzero weight space is one-dimensional.
    
    If the valid part of $D$ satisfies one of the conditions in Theorem~\ref{thm-lower-equivalent-condition}, then for any $C$ and $C'$ satisfying~\eqref{eq-lower-condition}, we have $\det\left(Y_D^C\right)=\det\left(Y_D^{C'}\right)$.
    Thus every nonzero weight space is one-dimensional, and $\chi_D(\mathbf{x})$ is zero-one.
    
    Conversely, if none of these conditions holds, the proof of Theorem~\ref{thm-lower-equivalent-condition} yields two diagrams $C$ and $C'$ satisfying~\eqref{eq-lower-condition} such that $\det\left(Y_D^C\right)$ and $\det\left(Y_D^{C'}\right)$ are not scalar multiples of each other. Therefore, the corresponding weight space has dimension at least two, and $\chi_D(\mathbf{x})$ is not zero-one.
	\end{proof}

    As a direct consequence of Theorem~\ref{thm-lower}, we obtain the following pattern-avoidance characterization of the Stanley symmetric functions that attain the lower bound, equivalently, whose monomial coefficients are all either $0$ or $1$.
    
\begin{cor}
The Stanley symmetric function $F_w$ attains the lower bound if and only if $w$ avoids the patterns
\[
321,\quad 2143,\quad 2413,\quad 3142,\quad 3412.
\]
\end{cor}
\begin{proof}
    For a Rothe diagram, the lower-bound criterion is equivalent to requiring that $D(w)$ be contained in a single row or a single column. Under the correspondence
$$
(i,j)\in \operatorname{Inv}(w)
\longleftrightarrow
(i,w_j)\in D(w),
$$
this means that all inversions of $w$ have the same left endpoint or the same right endpoint.
If all inversions have a common endpoint, then it is clear that $w$ avoids $321$.
Moreover, each permutation in
$\{2143,\, 2413,\, 3142,\,3412\}$
contains two disjoint inversions.

Conversely, suppose that $w$ avoids these five patterns. Then no two inversions are disjoint; otherwise, the pattern induced by their four endpoints would either contain $321$ or be one of $\{2143,\, 2413,\, 3142,\,3412\}$.
Thus the edge set of the inversion graph of $w$ is pairwise intersecting. A pairwise intersecting family of edges is either a star or consists of the three edges of a triangle. Since a triangle in the inversion graph corresponds to a $321$-pattern, all inversions of $w$ share a common endpoint. This endpoint cannot be both a left and a right endpoint, because
$$
a<p<b,
\qquad
w_a>w_p>w_b
$$
would produce a $321$-pattern. Therefore, all inversions have the same left endpoint or the same right endpoint, completing the proof.
\end{proof}
	
	\section{Proof of Theorem~\ref{thm-upper}}\label{sec-upper} 
	
	In this section, we proceed to prove Theorem~\ref{thm-upper}. To this end, we first introduce an order of fillings, which will be used to identify a unique maximal filling in each determinant expansion. 

    For a filling $F$, let ${\rm fix}_F$ denote the number of fixed points, where the element $i$ is called a fixed point if it is filled exactly in the $i$-th row, and for the remaining elements, let $f_i$ denote the total sum of entries filled in the $i$-th row.
	Define the weight of $F$ by
	$$
	\mathrm{wt}(F)=(f_1,\ldots,f_n).
	$$	
	Then we can define the rule of the filling order as follows.
	\begin{itemize}
		\item If ${\rm fix}_{F}<{\rm fix}_{F'}$, then we have $F<F'$.
		\item If ${\rm fix}_{F}$=${\rm fix}_{F'}$, and $\mathrm{wt}(F)<\mathrm{wt}(F')$ in lexicographical order, then we have $F<F'$.
	\end{itemize}


    Note that in the expansion of $\det(Y^C_D)=\sum_{F\in \mathcal{F}_D(C) }{\rm sgn}(F) y^{F}$, for a filling $F$, if there exists a distinct filling $F'$ such that $y^F = y^{F'}$ and $\operatorname{sgn}(F) = -\operatorname{sgn}(F')$, then we say that $y^F$ can be canceled in the determinant $\det(Y^C_D)$.
    For example, as illustrated in Figure~\ref{fig:exam-cancel}, let $R_i^1$ and $R_i^{2}$ be the multisets consisting of the entries in the $i$-th row of $F^{(1)}$ and $F^{(2)}$, respectively. It is clear that $y^{F^{(1)}}=y^{F^{(2)}}$, since $R_i^1=R_i^2$ for all $i\in[5]$. By the definition in \eqref{eq-inv}, we have ${\rm inv}(F^{(1)})=0+2+2=4$ and ${\rm inv}(F^{(2)})=3+1+1=5$, which yields ${\rm sgn}(F^{(1)})=1$ and ${\rm sgn}(F^{(2)})=-1$. Thus the two terms $y^{F^{(1)}}$ and $y^{F^{(2)}}$ cancel each other.
      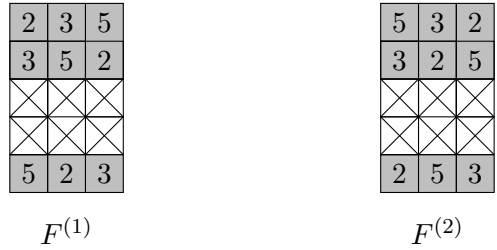
\begin{figure}[ht]
    \centering
    \begin{minipage}[b]{0.3\textwidth}
                \centering
        \begin{tikzpicture}
            \filldraw [lightgray] (0,0.5) rectangle (1.5,1.5);
            \filldraw [lightgray] (0,-0.5) rectangle (1.5,-1);
            \draw[black] (1+0.5,0+1.5) rectangle (1,0.5);
            \draw[black] (1+0.5,0) rectangle (1,0.5);
            \draw[black] (1+0.5,0+1) rectangle (1,0.5);
            \draw[black] (1+0.5,0-0.5) rectangle (1,0.5-0.5);
            \draw[black] (0,0) rectangle (0.5,0.5);
            \draw[black] (0.5,0) rectangle (1,0.5);
            \draw[black] (0,0.5) rectangle (0.5,1);
            \draw[black] (0.5,0.5) rectangle (1,1);
            \draw[black] (0,0.5) rectangle (0.5,1);
            \draw[black] (0,1) rectangle (0.5,1.5);
            \draw[black] (0.5,1) rectangle (1,1.5);
            \draw[black] (0.5,-0.5) rectangle (1,0.5);
            \draw[black] (0,-0.5) rectangle (0.5,0);
            \draw[black] (0,-0.5) rectangle (0.5,-1);
            \draw[black] (1,-0.5) rectangle (0.5,-1);
            \draw[black] (1,-0.5) rectangle (1.5,-1);
            \draw[black] (0.5,0.5) -- (1,0);
            \draw[black] (0,-0.5) -- (0.5,-0.5);
            \draw[black] (0,0) -- (0.5,-0.5);
            \draw[black] (0,-0.5) -- (0.5,0);
            \draw[black] (0.5,-0.5) -- (1,0);
            \draw[black] (0.5,0) -- (1,-0.5);
            \draw[black] (0,0) -- (0.5,0.5);
            \draw[black] (0,0.5) -- (0.5,0);
            \draw[black] (0.5,0) -- (1,0.5);
            \draw[black] (0+1,0) -- (0.5+1,-0.5);
            \draw[black] (0+1,-0.5) -- (0.5+1,0);
            \draw[black] (0+1,0.5) -- (0.5+1,-0.5+0.5);
            \draw[black] (0+1,-0.5+0.5) -- (0.5+1,0.5);
            \node at (0.25,0.75+.5) {2};
            \node at (0.25+.5,0.75+.5) {3};
            \node at (0.25+1,0.75+.5) {5};
            \node at (0.25,0.75) {3};
            \node at (0.25+.5,0.75) {5};
            \node at (0.25+1,0.75) {2};
            \node at (0.25,0.75-1.5) {5};
            \node at (0.25+.5,0.75-1.5) {2};
            \node at (0.25+1,0.75-1.5) {3};
            \node at (0.75,-1.5) {$F^{(1)}$};
        \end{tikzpicture}
    \end{minipage}
    \begin{minipage}[b]{0.3\textwidth}
                \centering
        \begin{tikzpicture}
            \filldraw [lightgray] (0,0.5) rectangle (1.5,1.5);
            \filldraw [lightgray] (0,-0.5) rectangle (1.5,-1);
            \draw[black] (1+0.5,0+1.5) rectangle (1,0.5);
            \draw[black] (1+0.5,0) rectangle (1,0.5);
            \draw[black] (1+0.5,0+1) rectangle (1,0.5);
            \draw[black] (1+0.5,0-0.5) rectangle (1,0.5-0.5);
            \draw[black] (0,0) rectangle (0.5,0.5);
            \draw[black] (0.5,0) rectangle (1,0.5);
            \draw[black] (0,0.5) rectangle (0.5,1);
            \draw[black] (0.5,0.5) rectangle (1,1);
            \draw[black] (0,0.5) rectangle (0.5,1);
            \draw[black] (0,1) rectangle (0.5,1.5);
            \draw[black] (0.5,1) rectangle (1,1.5);
            \draw[black] (0.5,-0.5) rectangle (1,0.5);
            \draw[black] (0,-0.5) rectangle (0.5,0);
            \draw[black] (0,-0.5) rectangle (0.5,-1);
            \draw[black] (1,-0.5) rectangle (0.5,-1);
            \draw[black] (1,-0.5) rectangle (1.5,-1);
            \draw[black] (0.5,0.5) -- (1,0);
            \draw[black] (0,-0.5) -- (0.5,-0.5);
            \draw[black] (0,0) -- (0.5,-0.5);
            \draw[black] (0,-0.5) -- (0.5,0);
            \draw[black] (0.5,-0.5) -- (1,0);
            \draw[black] (0.5,0) -- (1,-0.5);
            \draw[black] (0,0) -- (0.5,0.5);
            \draw[black] (0,0.5) -- (0.5,0);
            \draw[black] (0.5,0) -- (1,0.5);
            \draw[black] (0+1,0) -- (0.5+1,-0.5);
            \draw[black] (0+1,-0.5) -- (0.5+1,0);
            \draw[black] (0+1,0.5) -- (0.5+1,-0.5+0.5);
            \draw[black] (0+1,-0.5+0.5) -- (0.5+1,0.5);
            \node at (0.25,0.75+.5) {5};
            \node at (0.25+.5,0.75+.5) {3};
            \node at (0.25+1,0.75+.5) {2};
            \node at (0.25,0.75) {3};
            \node at (0.25+.5,0.75) {2};
            \node at (0.25+1,0.75) {5};
            \node at (0.25,0.75-1.5) {2};
            \node at (0.25+.5,0.75-1.5) {5};
            \node at (0.25+1,0.75-1.5) {3};
            \node at (0.75,-1.5) {$F^{(2)}$};
        \end{tikzpicture}
    \end{minipage}
    \caption{Two fillings can cancel each other in~\eqref{eq-filling-det}, i.e., $y^{F^{(1)}}=-y^{F^{(2)}}$.}
    \label{fig:exam-cancel}
    \end{figure}

    Next we consider the maximum filling in this order as follows.

    \begin{lem}\label{exist}
		Given two diagrams $D=(D_1,D_2,\ldots,D_n)$ and $C=(C_1,C_2,\ldots,C_n)$ of $[n]^2$ with $|C_j|=|D_j|$ for each $j\in[n]$, there exists a unique filling $F_{led}\in\mathcal{F}_D(C)$, which is maximum in the filling order, i.e., $F_{led}\ge F$ for any $F\in\mathcal{F}_{D}(C)$. Moreover, let $y^{F_{led}}$ denote the monomial corresponding to $F_{led}$, then this term cannot be canceled in the determinant $\det\left(Y^C_D\right)$.
	\end{lem}
	\begin{proof}
		The filling $F_{led}$ can be constructed as follows. First, fill the element $i$ in the $(i,j)$-th box for each $i\in C_j\cap D_j$. This maximizes the number of fixed points.
		Then, for each $j\in[n]$, from the top to the bottom, we fill the remaining elements of $C_j$ in the remaining rows corresponding to the elements of $D_j$ in decreasing order. 
		  
		Now we claim that $F_{led}$ is unique and maximum in the filling order. Consider a filling $F\ge F_{led}$ in $\mathcal{F}_D(C)$. By the construction of $F_{led}$, we have that ${\rm fix}_{F_{led}}$ is as large as possible. So we have to fill each element $i\in C_j\cap D_j$ in the box $(i,j)$ in the same way as $F_{led}$ does, which means that ${\rm fix}_F={\rm fix}_{led}$.
		Then, in $F$, we need to consider the filling of the remaining elements. Without loss of generality, for a fixed column $j$, assume that the first row of $F$ is nonempty. 
        To obtain a filling with larger weight, according to the construction process of $F_{led}$, we fill in the box $(i,j)$ with the largest remaining element in $C_j$, so that the elements in $F$ cannot be larger than those in $F_{led}$. Continue this process until all boxes are filled, and we can finally conclude that $F=F_{led}$.

        To prove that the leading monomial $y^{F_{led}}$ cannot be canceled, we proceed by contradiction. Assume that there exists a filling $F'$ such that $y^{F'} = -y^{F_{\text{led}}}$. 
        Then the number of factors $y_{ii}$ in both monomial $y^{F'}$ and $y^{F_{led}}$ are equal for all $i$, and therefore in $F'$, each element $i\in C_j\cap D_j$ is also filled in the $(i,j)$-th box for all possible $j$ according to the construction of $F_{led}$. This implies that ${\rm fix}_{F'}={\rm fix}_{F_{led}}$. 
        Let $R^i$ and $R^i_{led}$ be the multisets consisting of entries on the $i$-th row of $F'$ and $F_{led}$ respectively. Since each entry $a$ in row $i$ corresponds to the indeterminates $y_{a,i}$ in both $y^{F'}$ and $y^{F_{led}}$, the condition of $y^{F'}=y^{F_{led}}$ ensures that $R^i=R^i_{led}$ for all $i$, which leads to $\mathrm{wt}(F')=\mathrm{wt}(F_{led})$. Figure~\ref{fig:same-wt} gives an example with $D = (\{2,4\}, \{2,4\})$ and $C = (\{4,5\}, \{3,4\})$, showing a leading filling and a distinct filling having the same weight.  
        However, when ${\rm fix}(F_{led})={\rm fix}(F')$ is maximal, the weights of $F_{led}$ and $F'$ are equal only if $F_{led}=F'$.
        That is, by the construction process of $F_{led}$, the non-fixed entries of $F'$ must also be filled decreasingly from the top to the bottom as mentioned before.   
        In this way, we obtain that $F'=F_{led}$, which leads to a contradiction. This completes the proof.      	
        \end{proof}

 \begin{figure}[ht]
         \centering
         \begin{minipage}[b]{0.3\textwidth}
         	\centering
		\begin{tikzpicture}
			\draw[black] (0,-1.5) rectangle (0.5,-1);
            \draw[black] (0.5,-1.5) rectangle (1,-1);
            \filldraw[lightgray] (0,-1) rectangle (0.5,-0.5);
            \filldraw[lightgray] (0,0) rectangle (0.5,0.5);
            \filldraw[lightgray] (0.5,0) rectangle (1,0.5);
            \filldraw[lightgray] (0.5,-1) rectangle (1,-0.5);
            \draw[black] (0.5,-1) rectangle (1,-0.5);
            \draw[black] (0,-1) rectangle (0.5,-0.5);
            \draw[black] (0,-0.5) rectangle (0.5,0);            
            \draw[black] (0.5,-0.5) rectangle (1,0);
			\draw[black] (0,0) rectangle (0.5,0.5);
            \draw[black] (0.5,0) rectangle (1,0.5);
            \draw[black] (0.5,0.5) rectangle (1,1);
            \draw[black] (0,0.5) rectangle (0.5,1);
            \draw[black] (0,-1.5) -- (0.5,-1);
            \draw[black] (0,-1) -- (0.5,-1.5);
            \draw[black] (0,-0.5) -- (0.5,0);
            \draw[black] (0,0) -- (0.5,-0.5);
            \draw[black] (0.5,-1.5) -- (1,-1);
            \draw[black] (0.5,-1) -- (1,-1.5);
            \draw[black] (0.5,-0.5) -- (1,0);
            \draw[black] (0.5,0) -- (1,-0.5);
            \draw[black] (0.5,-0.5+1.5) -- (1,0+0.5);
            \draw[black] (0.5,0.5) -- (1,-0.5+1.5);
            \draw[black] (-0.5+0.5,-0.5+1.5) -- (1-0.5,0+0.5);
            \draw[black] (-0.5+0.5,0.5) -- (1-.5,-0.5+1.5);
            \node at (0.75,0.25) {3};
            \node at (0.75,-0.75) {4};
            \node at (0.25,0.25) {5};
            \node at (0.25,-0.75) {4};
            \node at (0.5,-2) {$F_{led}$};
		\end{tikzpicture}
        \end{minipage}\hspace{0.02\textwidth}
        \hspace{0.05\textwidth}
        \begin{minipage}[b]{0.3\textwidth}
        \centering
        \begin{tikzpicture}
			\draw[black] (0,-1.5) rectangle (0.5,-1);
            \draw[black] (0.5,-1.5) rectangle (1,-1);
            \filldraw[lightgray] (0,-1) rectangle (0.5,-0.5);
            \filldraw[lightgray] (0,0) rectangle (0.5,0.5);
            \filldraw[lightgray] (0.5,0) rectangle (1,0.5);
            \filldraw[lightgray] (0.5,-1) rectangle (1,-0.5);
            \draw[black] (0.5,-1) rectangle (1,-0.5);
            \draw[black] (0,-1) rectangle (0.5,-0.5);
            \draw[black] (0,-0.5) rectangle (0.5,0);            
            \draw[black] (0.5,-0.5) rectangle (1,0);
			\draw[black] (0,0) rectangle (0.5,0.5);
            \draw[black] (0.5,0) rectangle (1,0.5);
            \draw[black] (0.5,0.5) rectangle (1,1);
            \draw[black] (0,0.5) rectangle (0.5,1);
            \draw[black] (0,-1.5) -- (0.5,-1);
            \draw[black] (0,-1) -- (0.5,-1.5);
            \draw[black] (0,-0.5) -- (0.5,0);
            \draw[black] (0,0) -- (0.5,-0.5);
            \draw[black] (0.5,-1.5) -- (1,-1);
            \draw[black] (0.5,-1) -- (1,-1.5);
            \draw[black] (0.5,-0.5) -- (1,0);
            \draw[black] (0.5,0) -- (1,-0.5);
            \draw[black] (0.5,-0.5+1.5) -- (1,0+0.5);
            \draw[black] (0.5,0.5) -- (1,-0.5+1.5);
            \draw[black] (-0.5+0.5,-0.5+1.5) -- (1-0.5,0+0.5);
            \draw[black] (-0.5+0.5,0.5) -- (1-.5,-0.5+1.5);
            \node at (0.75,0.25) {4};
            \node at (0.75,-0.75) {3};
            \node at (0.25,0.25) {4};
            \node at (0.25,-0.75) {5};
            \node at (0.5,-2) {$F'$};
		\end{tikzpicture}
        \end{minipage}
         \caption{A leading filling and another filling having the same weight.}
         \label{fig:same-wt}
     \end{figure}
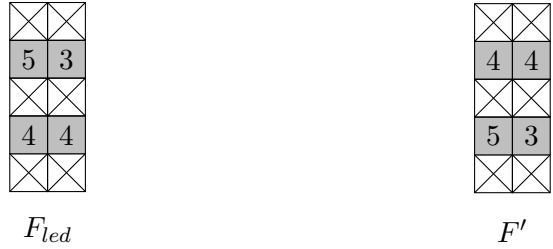

   For the maximum fillings illustrated above, we give the following statement of the corresponding leading monomials.
	\begin{lem}\label{leading}
		Let $D$ be a diagram that avoids the configurations in Figure~\ref{fig-upper} and let $C^{(1)},\ldots,C^{(m)}$ be all the diagrams $C=(C_1,C_2,\ldots,C_n)$ with $|C_j|=|D_j|$ for $j\in[n]$ and ${\bf x}^{C}={\bf x}^a$, where ${\bf x}^a$ is a monomial in $\chi_{D}({\bf x})$. For simplicity, write ${\bf y}^{(a)}={y}^{F_{led}^{C^{(a)}}}$  Then we have the following results:
		\begin{enumerate}
			\item[(a)] For $1\leq a \leq m$, the leading monomial ${\bf y}^{(a)}$ appears in the polynomial $\det\left(Y_D^{C^{(a)}}\right)$.
			
			\item[(b)] The leading monomials $\{{\bf y}^{(a)}\}$ are distinct.
		\end{enumerate}
	\end{lem}
	
	\begin{proof}
		Item (a) follows from the fact that ${\bf y}^{(a)}$ cannot be canceled due to Lemma~\ref{exist}.
        For item (b), we proceed by contradiction. Assume that there exist $1\le a\neq b\le m$ such that ${\bf y}^{(a)}={\bf y}^{(b)}$. In this case, the multiset of entries in row $i$ of the corresponding fillings $F_{\text{led}}^{(a)}$ and $F_{\text{led}}^{(b)}$ must be identical. Now consider any \(1 \leq i \leq n\) such that the \(i\)-th row of \(F_{\mathrm{led}}^{(a)}\) differs from the \(i\)-th row of \(F_{\mathrm{led}}^{(b)}\). Then there exists an element \(k \neq i\) whose column position in the \(i\)-th row is different in \(F_{\mathrm{led}}^{(a)}\) and \(F_{\mathrm{led}}^{(b)}\). In other words, there exist distinct columns \(j \neq j'\) such that \(k\) appears in column \(j\) of the \(i\)-th row of \(F_{\mathrm{led}}^{(a)}\) and in column \(j'\) of the \(i\)-th row of \(F_{\mathrm{led}}^{(b)}\). 

        To avoid the configurations in Figure~\ref{fig-upper}, at least one of the boxes $(k,j)$ or $(k,j')$ must lie in the diagram $D$, and we may assume that the box $(k,j)$ exists. By the construction of the leading filling, we place the entry $k$ in the $k$-th row if $k\in C_j\cap D_j$, so the entry $k$ must be filled in the box $(k,j)$ with $k\neq i$. This yields a contradiction and completes the proof.
	\end{proof}
	
	Now we are in a position to give the proof of Theorem~\ref{thm-upper}.

	\begin{proof}[Proof of Theorem~\ref{thm-upper}]
		The character $\chi_D({\bf x})$ attains the upper bound of \eqref{eq-ineq-x} if and only if the generators of the module $\mathbb{W}_D$ are linearly independent. 
        For the sufficiency,  we want to show that if $D$ avoids the configurations in Figure~\ref{fig-upper}, then the generators in $\{\det(Y^{C^{(i)}}_D)\}$ are linearly independent. Following the idea in the proof of \cite[Theorem 2.5]{PengLinSun2024} and \cite[Thm 2.5.8]{2001The}, consider the following linear relation
		$$c_1\det\left(Y^{C^{(1)}}_{D}\right)+c_2\det\left(Y^{C^{(2)}}_{D}\right)+\cdots+c_m\det\left(Y^{C^{(m)}}_{D}\right)=0.
		$$
		Without loss of generality, we choose ${\bf y}^{(1)}$ as a maximal monomial among the leading monomials $\{y^{(a)}\}$ with respect to the order of fillings. By Lemma~\ref{leading} and the maximality of $y^{(1)}$, we can deduce that ${\bf y}^{(1)}$ only appears in $\det\left(Y^{C^{(1)}}_{D}\right)$, so that $c_1=0$. Continuing this process, we can conclude that $c_i=0$ for all $2\leq i\le m$. Thus the generators are linearly independent.
		
		Conversely, assume that the diagram $D$ contains at least one configuration in Figure~\ref{fig-upper}. Since Schur characters are stable under the column permutations, we can only consider the case where those two columns containing this configuration are adjacent; denote them by $j_0$ and $j_0+1$. For each row in these two columns, there are four possible configurations, as illustrated in Figure~\ref{fig-ex-two-columns}.

        \begin{figure}[ht]
        \centering
        \begin{minipage}[b]{0.2\textwidth}
        \centering
        \begin{tikzpicture}
			\filldraw [lightgray] (0.5,0) rectangle (1,0.5);
			\draw[black] (0.5,0) rectangle (1,0.5);
            \draw[black] (0,0) rectangle (0.5,0.5);
            \draw[black] (0,0) -- (0.5,0.5);
			\draw[black] (0,0.5) -- (0.5,0);
            \node at (0.5,-0.75) {(I)};
		\end{tikzpicture}
        \end{minipage}
        \begin{minipage}[b]{0.2\textwidth}
        \centering
        \begin{tikzpicture}
			\filldraw [lightgray] (0.5,0) rectangle (1,0.5);
			\draw[black] (0.5,0) rectangle (1,0.5);
            \filldraw [lightgray] (0,0) rectangle (0.5,0.5);
            \draw[black] (0,0) rectangle (0.5,0.5);
            \node at (0.5,-0.75) {(II)};
		\end{tikzpicture}
        \end{minipage}
        \begin{minipage}[b]{0.2\textwidth}
        \centering
        \begin{tikzpicture}
			\filldraw [lightgray] (0,0) rectangle (0.5,0.5);
			\draw[black] (0.5,0) rectangle (1,0.5);
            \draw[black] (0,0) rectangle (0.5,0.5);
            \draw[black] (0.5,0) -- (1,0.5);
			\draw[black] (0.5,0.5) -- (1,0);
            \node at (0.5,-0.75) {(III)};
		\end{tikzpicture}
        \end{minipage}
        \begin{minipage}[b]{0.2\textwidth}
        \centering
        \begin{tikzpicture}
			\draw[black] (0.5,0) rectangle (1,0.5);
            \draw[black] (0,0) rectangle (0.5,0.5);
            \draw[black] (0,0) -- (0.5,0.5);
			\draw[black] (0,0.5) -- (0.5,0);
            \draw[black] (0.5,0) -- (1,0.5);
			\draw[black] (0.5,0.5) -- (1,0);
            \node at (0.5,-0.75) {(IV)};
		\end{tikzpicture}
        \end{minipage}
		    \caption{Four possible parts in two adjacent columns.}
		    \label{fig-ex-two-columns}
		\end{figure}
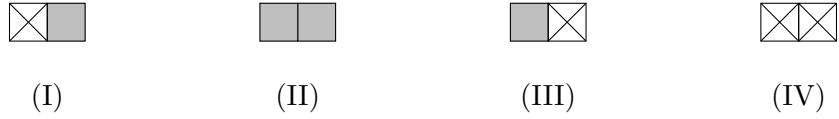
        
        We may further permute the rows so that all rows of Configuration~(I) appear at the top, followed in order by rows of Configuration~(II), then (IV), with rows of Configuration~(III) at the bottom. After a suitable row permutation (and, since Schur characters are stable under column permutations, with the two columns taken adjacent), each column of the two-column subdiagram occupies a consecutive set of rows; thus it can be viewed as a skew Young diagram $\lambda/\mu$. See Figure~4.4 for an example.
        
  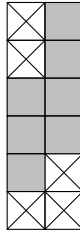
\begin{figure}[ht]
         \centering
		\begin{tikzpicture}
			\centering
			\draw[black] (0,-1.5) rectangle (0.5,-1);
            \draw[black] (0.5,-1.5) rectangle (1,-1);
            \draw[black] (0.5,-1) rectangle (1,-0.5);
            \filldraw[lightgray] (0,-1) rectangle (0.5,-0.5);
            \filldraw[lightgray] (0,-0.5) rectangle (0.5,0);
            \filldraw[lightgray] (0.5,-0.5) rectangle (1,0);
            \filldraw[lightgray] (0,0) rectangle (0.5,0.5);
            \filldraw[lightgray] (0.5,0) rectangle (1,0.5);
            \filldraw[lightgray] (0.5,0.5) rectangle (1,1);
            \filldraw[lightgray] (0.5,1) rectangle (1,1.5);
            \draw[black] (0,-1) rectangle (0.5,-0.5);
            \draw[black] (0,-0.5) rectangle (0.5,0);            
            \draw[black] (0.5,-0.5) rectangle (1,0);
			\draw[black] (0,0) rectangle (0.5,0.5);
            \draw[black] (0.5,0) rectangle (1,0.5);
            \draw[black] (0.5,0.5) rectangle (1,1);
            \draw[black] (0.5,1) rectangle (1,1.5);
            \draw[black] (0,0.5) rectangle (0.5,1);
            \draw[black] (0,1) rectangle (0.5,1.5);
            \draw[black] (0,-1.5) -- (0.5,-1);
            \draw[black] (0,-1) -- (0.5,-1.5);
            \draw[black] (0,0.5) -- (0.5,1);
            \draw[black] (0,1) -- (0.5,0.5);
            \draw[black] (0,1) -- (0.5,1.5);
            \draw[black] (0,1.5) -- (0.5,1);
            \draw[black] (0.5,-1.5) -- (1,-1);
            \draw[black] (0.5,-1) -- (1,-1.5);
            \draw[black] (0.5,-1) -- (1,-0.5);
            \draw[black] (0.5,-0.5) -- (1,-1);
		\end{tikzpicture}
         \caption{An instance of the combination of four configurations in Figure \ref{fig-ex-two-columns}.}
         \label{fig:exam-combination}
     \end{figure}

		
		Next consider the family $\mathcal{C}$ of diagrams $C$ satisfying that
		\begin{itemize}
			\item $|C_j|=|D_j|$ for $1\le j\le n$;
			\item $C_j=D_j$ for $j\neq j_0,j_0+1$.
		\end{itemize}
		Then we have
		\begin{align}\label{det}
			\det\left(Y^C_D\right)=\det\left(Y_{D_{j_0}}^{C_{j_0}}\right)\cdot \det\left(Y_{D_{j_0+1}}^{C_{j_0+1}}\right)\cdot \prod_{j\neq j_0,j_0+1}\det\left( Y_{D_j}^{C_j} \right)
		\end{align}
		with the last term fixed. 
        Consequently, the generators in $\left\{\det\left(Y^C_D\right)\right\}$ are linearly dependent if and only if the generators in
        \[\left\{\det\left(Y_{D(j_0,j_0+1)}^{C(j_0,j_0+1)}\right)=\det\left(Y_{D_{j_0}}^{C_{j_0}}\right)\cdot \det\left(Y_{D_{j_0+1}}^{C_{j_0+1}}\right)\right\}\]
        are linearly dependent, where $D(i,j)$ denotes the subdiagram consisting of the $i$-th and $j$-th columns of a diagram $D$. So $\chi_D$ attains the upper bound only if $\chi_{D(j_0,j_0+1)}$ attains the upper bound.
    
        For a diagram  $D'=(D'_1,\ldots,D'_n)$, by the upper bound in \eqref{eq-ineq-x}, together with its principal specialization~\eqref{eq-ineq-1}, the character \(\chi_{D'}\) attains the upper bound if and only if $\chi_{D'}(1,\ldots,1)$ equals the number of diagrams $C'=(C'_1,\ldots,C'_n)$ satisfying $|C'_j|=|D'_j|$ for all $j\in [n]$.
        Since $D(j_0,j_0+1)$ is a Young diagram (as stated previously), by Theorem~\ref{them-Schur-Schur module}, the Schur character $\chi_{D(j_0,j_0+1)}$ is in fact a skew Schur polynomial. By the definition of skew Schur polynomials (illustrated in~\eqref{eq-def-Schur}), each summand in a skew Schur polynomial corresponds bijectively to a semistandard Young tableau. 
        To be clear, we refer the readers to the following identities:
        \begin{align}\label{bij-iden}
            \left|\{C'\colon |C_j'|=|D_j'| ~\text{for all $j$}\}\right|  &\xlongequal[\text{by (\ref{eq-ineq-x})}]{\text{if $D'$ attains the upper bound}}\chi_{D'}(1,\ldots,1) \notag\\
            &\xlongequal[\text{by Theorem \ref{them-Schur-Schur module}} ]
            {\text{if $D'=D(\lambda/\mu)$}}s_{\lambda/\mu}(1,\ldots,1)\notag\\
        &\xlongequal{\text{by definition}}\left|\{T\colon T\text{ is an SSYT of shape }\lambda/\mu \}\right| .
        \end{align}
  
        Since $D(j_0,j_0+1)$ is clearly a skew Young diagram (as stated previously), to show that $D(j_0,j_0+1)$ does not attain the upper bound, it suffices to show that the natural map from semistandard Young tableaux of shape $\lambda/\mu$ to the admissible diagrams $C(j_0,j_0+1)$ is injective but not surjective. 
        The map is naturally constructed by considering the elements filled in the column $j_0$ (resp. $j_0+1$) in $T$ as elements of the set $C_{j_0}$ (resp. $C_{j_0+1}$).
        For example, the Young tableau depicted in Figure~\ref{fig:exam-combination-filling} is mapped to $(C_1,C_2)=(\{3,5,6\},\{1,2,3,6\})$.
		
		Clearly the map is an injection. Next we show that this map is not surjective, which means that for a subdiagram $C(j_0,j_0+1)$, there is no semistandard Young tableau corresponding to such a diagram.  
        For $D(j_0,j_0+1)=(\left\{i_1,i_1+1,\ldots,i_{k_1}\right\},\left\{1,2,\ldots,i_{k_2}\right\})$, we can construct 
        $$C(j_0,j_0+1)=(\{n-i_{k_1}+i_1,n-i_{k_1}+i_1+1,\ldots,n\},\{1,2,\ldots,i_{k_2}\}).$$ 
        In any semistandard Young tableau of shape $D(j_0,j_0+1)$, the entries strictly increase down each column. The left column is $C_{j_0}=\{n-i_{k_1}+i_1,\dots,n\}$, whose smallest element $n-i_{k_1}+i_1$ must therefore occupy the top box $(i_1,j_0)$. The right column is $C_{j_0+1}=\{1,2,\dots,i_{k_2}\}$, which contains the entry $i_1$ and hence also places $i_1$ in some box of column $j_0+1$, in particular in row $i_1$ as $(i_1,j_0+1)$. Since $n-i_{k_1}+i_1>i_1$, the row-weakly-increasing condition of an SSYT is violated at row $i_1$. Thus no semistandard Young tableau $T$ satisfies $\phi(T)=C(j_0,j_0+1)$, so $\phi$ is not surjective. This completes the proof.

	\end{proof}
	
	  \begin{figure}[ht]
         \centering
		\begin{tikzpicture}
			\centering
			\draw[black] (0,-1.5) rectangle (0.5,-1);
            \draw[black] (0.5,-1.5) rectangle (1,-1);
            \draw[black] (0.5,-1) rectangle (1,-0.5);
            \filldraw[lightgray] (0,-1) rectangle (0.5,-0.5);
            \filldraw[lightgray] (0,-0.5) rectangle (0.5,0);
            \filldraw[lightgray] (0.5,-0.5) rectangle (1,0);
            \filldraw[lightgray] (0,0) rectangle (0.5,0.5);
            \filldraw[lightgray] (0.5,0) rectangle (1,0.5);
            \filldraw[lightgray] (0.5,0.5) rectangle (1,1);
            \filldraw[lightgray] (0.5,1) rectangle (1,1.5);
            \draw[black] (0,-1) rectangle (0.5,-0.5);
            \draw[black] (0,-0.5) rectangle (0.5,0);            
            \draw[black] (0.5,-0.5) rectangle (1,0);
			\draw[black] (0,0) rectangle (0.5,0.5);
            \draw[black] (0.5,0) rectangle (1,0.5);
            \draw[black] (0.5,0.5) rectangle (1,1);
            \draw[black] (0.5,1) rectangle (1,1.5);
            \draw[black] (0,0.5) rectangle (0.5,1);
            \draw[black] (0,1) rectangle (0.5,1.5);
            \draw[black] (0,-1.5) -- (0.5,-1);
            \draw[black] (0,-1) -- (0.5,-1.5);
            \draw[black] (0,0.5) -- (0.5,1);
            \draw[black] (0,1) -- (0.5,0.5);
            \draw[black] (0,1) -- (0.5,1.5);
            \draw[black] (0,1.5) -- (0.5,1);
            \draw[black] (0.5,-1.5) -- (1,-1);
            \draw[black] (0.5,-1) -- (1,-1.5);
            \draw[black] (0.5,-1) -- (1,-0.5);
            \draw[black] (0.5,-0.5) -- (1,-1);
            \node at (0.75,1.25) {1};
            \node at (0.75,0.75) {2};
            \node at (0.75,0.25) {3};
            \node at (0.75,-0.25) {6};
            \node at (0.25,0.25) {3};
            \node at (0.25,-0.25) {5};
            \node at (0.25,-0.75) {6};
		\end{tikzpicture}
         \caption{An instance of the filling in Figure~\ref{fig:exam-combination}.}
         \label{fig:exam-combination-filling}
     \end{figure}

Similarly, we can also obtain the pattern avoidance condition for the Stanley symmetric functions to reach their upper bounds by applying Theorem~\ref{thm-upper}.

\begin{cor}
    The Stanley symmetric function $F_w$ attains its upper bound if and only if $w$ avoids the patterns $312$ and $321$.
\end{cor}
\begin{proof}
    By Theorem~\ref{thm-upper}, for an arbitrary diagram, equality with the upper bound in \ref{eq-ineq-x} is equivalent to avoiding the two forbidden $2\times2$ configurations consisting of two boxes in one row and two empty cells in another. For a Rothe diagram, this is equivalent to requiring that every row contain at most one box. 
    Indeed, suppose that row $i$ contains boxes in two distinct columns $a$ and $b$. Then $i<w^{-1}(a)$ and $i<w^{-1}(b)$. Setting $r=\max\{w^{-1}(a),w^{-1}(b)\}$, we have $(r,a),(r,b)\notin D(w)$. Thus these two columns form a forbidden configuration. The converse is immediate.
    
    Moreover, row $i$ contains at least two boxes if and only if there exist $i<j<k$ such that $w_i>w_j$ and $w_i>w_k$. If $w_j>w_k$, then $w_i,w_j,w_k$ form a $321$-pattern; otherwise, they form a $312$-pattern. Conversely, every occurrence of $312$ or $321$ yields two boxes in the row corresponding to its first entry. Therefore, every row of $D(w)$ contains at most one box if and only if $w$ avoids $312$ and $321$. This completes the proof.
\end{proof}

	\vskip 0.5cm
	\noindent \textbf{Acknowledgments.} Zhuowei Lin was supported by the Fundamental Research Funds for the Central Universities (No. 63263094) and the Natural Science Foundation of China (No. 12371329).

\bibliography{references}

@article{AS-1,
  author = {Assaf, S. and Searles, D.},
  title = {Kohnert tableaux and a lifting of quasi-Schur functions},
  journal = {J. Combin. Theory Ser. A},
  volume = {156},
  year = {2018},
  pages = {85--118}
}

@article{BH,
  author = {Br\"and\'en, P. and Huh, J.},
  title = {Lorentzian polynomials},
  journal = {Ann. of Math.},
  volume = {192},
  year = {2020},
  pages = {821--891}
}

@article{CFY,
  author = {Chen, Y. and Fan, N. J. Y. and Ye, Z.},
  title = {Zero-one Grothendieck polynomials},
  journal = {Sci. China Math.},
  volume = {69},
  number = {1},
  year = {2026},
  pages = {269--284}
}

@article{fan2022upper,
  author = {Fan, N. J. Y. and Guo, P. L.},
  title = {Upper bounds of Schubert polynomials},
  journal = {Sci. China Math.},
  volume = {65},
  year = {2022},
  pages = {1319--1330}
}

@article{2018Schubert,
  author = {Fink, A. and M\'esz\'aros, K. and St. Dizier, A.},
  title = {Schubert polynomials as integer point transforms of generalized permutahedra},
  journal = {Adv. Math.},
  volume = {332},
  year = {2018},
  pages = {465--475}
}

@article{fink2021zero,
  author = {Fink, A. and M\'esz\'aros, K. and St. Dizier, A.},
  title = {Zero-one Schubert polynomials},
  journal = {Math. Z.},
  volume = {297},
  year = {2021},
  pages = {1023--1042}
}

@article{fomin,
  author = {Fomin, S. and Greene, C. and Reiner, V. and Shimozono, M.},
  title = {Balanced labellings and Schubert polynomials},
  journal = {European J. Combin.},
  volume = {18},
  number = {4},
  year = {1997},
  pages = {373--389}
}

@book{Fulton,
  author = {Fulton, W. and Harris, J.},
  title = {Representation Theory},
  series = {Graduate Texts in Mathematics},
  volume = {129},
  publisher = {Springer-Verlag},
  address = {New York},
  year = {1991}
}

@article{GL,
  author = {Guo, P. L. and Lin, Z.},
  title = {Schubert polynomials and patterns in permutations},
  journal = {Math. Z.},
  volume = {311},
  year = {2026},
  pages = {Paper No. 16}
}

@article{GLP,
  author = {Guo, P. L. and Lin, Z. and Peng, S. C. Y.},
  title = {Zero-one dual characters of flagged Weyl modules},
  journal = {Canad. J. Math.},
  year = {2025},
  pages = {1--23},
  doi = {10.4153/S0008414X2510120X}
}

@article{HY-1,
  author = {Hodges, R. and Yong, A.},
  title = {Coxeter combinatorics and spherical Schubert geometry},
  journal = {J. Lie Theory},
  volume = {32},
  year = {2022},
  pages = {447--474}
}

@article{hodges2023multiplicity,
  author = {Hodges, R. and Yong, A.},
  title = {Multiplicity-free key polynomials},
  journal = {Ann. Comb.},
  volume = {27},
  year = {2023},
  pages = {387--411}
}

@article{Huh,
  author = {Huh, J. and Matherne, J. P. and M\'esz\'aros, K. and St. Dizier, A.},
  title = {Logarithmic concavity of Schur and related polynomials},
  journal = {Trans. Amer. Math. Soc.},
  volume = {375},
  year = {2022},
  pages = {4411--4427}
}

@article{Koh,
  author = {Kohnert, A.},
  title = {Weintrauben, Polynome, Tableaux},
  journal = {Bayreuth Math. Schrift.},
  volume = {38},
  year = {1990},
  pages = {1--97}
}

@article{LascouxSchutzenberger1985,
  author = {Lascoux, A. and Sch\"utzenberger, M.-P.},
  title = {Schubert polynomials and the Littlewood--Richardson rule},
  journal = {Lett. Math. Phys.},
  volume = {10},
  number = {2--3},
  year = {1985},
  pages = {111--124}
}

@phdthesis{Liu2010,
  author = {Liu, R. I.},
  title = {Specht modules and Schubert varieties for general diagrams},
  school = {Massachusetts Institute of Technology},
  address = {ProQuest LLC, Ann Arbor, MI},
  year = {2010}
}

@article{magyar1998schubert,
  author = {Magyar, P.},
  title = {Schubert polynomials and Bott--Samelson varieties},
  journal = {Comment. Math. Helv.},
  volume = {73},
  year = {1998},
  pages = {603--636}
}

@article{Magyar-1998-2,
  author = {Magyar, P.},
  title = {Borel-Weil theorem for configuration varieties and Schur modules},
  journal = {Adv. Math.},
  volume = {134},
  number = {2},
  year = {1998},
  pages = {328--366}
}

@article{mestrans,
  author = {M\'esz\'aros, K. and Setiabrata, L. and St. Dizier, A.},
  title = {An orthodontia formula for Grothendieck polynomials},
  journal = {Trans. Amer. Math. Soc.},
  volume = {375},
  number = {2},
  year = {2022},
  pages = {1281--1303}
}

@article{meszaros2021principal,
  author = {M\'esz\'aros, K. and St. Dizier, A. and Tanjaya, A.},
  title = {Principal specialization of dual characters of flagged Weyl modules},
  journal = {Electron. J. Combin.},
  volume = {28},
  year = {2021},
  pages = {Paper No. 4.17, 12 pp.}
}

@article{PengLinSun2024,
  author = {Peng, S. C. Y. and Lin, Z. and Sun, S. C. C.},
  title = {Upper bounds of dual flagged Weyl characters},
  journal = {Adv. Appl. Math.},
  volume = {160},
  year = {2024},
  pages = {102752},
  note = {12 pp.}
}

@book{2001The,
  author = {Sagan, B. E.},
  title = {The symmetric group. Representations, combinatorial algorithms, and symmetric functions},
  edition = {Second edition},
  series = {Graduate Texts in Mathematics},
  volume = {203},
  publisher = {Springer-Verlag},
  address = {New York},
  year = {2001}
}

@article{Stanley1984,
  author = {Stanley, R. P.},
  title = {On the number of reduced decompositions of elements of Coxeter groups},
  journal = {European J. Combin.},
  volume = {5},
  number = {4},
  year = {1984},
  pages = {359--372}
}

@misc{WYZZZ,
  author = {Wang, B. and Yang, A. L. B. and Zhang, C. X. T. and Zhang, P. B. and Zhang, Z.-X.},
  title = {Minkowski sum decomposition of Newton polytopes of skew Schur polynomials and beyond},
  note = {in preparation}
}

\end{document}